\documentclass[a4paper]{amsart}
\usepackage[leqno]{amsmath}
\usepackage{enumerate}
\usepackage{ifthen}
\usepackage{amstext,amsbsy,amsopn,amsthm,amsgen}
\usepackage{amsfonts,amscd,amsxtra,upref}
\usepackage{mathrsfs}
\usepackage{euscript}
\usepackage{amssymb}
\usepackage{hyperref}

\swapnumbers
\theoremstyle{plain}
\newtheorem{Thm}{Theorem}[section]
\newtheorem{Lem}[Thm]{Lemma}
\newtheorem{Cor}[Thm]{Corollary}
\newtheorem{Pro}[Thm]{Proposition}
\theoremstyle{definition}
\newtheorem{Def}[Thm]{Definition}
\newtheorem{Exm}[Thm]{Example}
\theoremstyle{remark}
\newtheorem{Rem}[Thm]{Remark}

\numberwithin{equation}{section}

\newcommand{\ITE}[3]{\ifthenelse{#1}{#2}{#3}}\newcommand{\ITEE}[4][]{\ITE{\equal{#2}{#3}}{#4}{#1}}
 \ITE{\isundefined{\texorpdfstring}}{\newcommand{\texorpdfstring}[2]{#1}}{}

\newenvironment{cor}[2][]{\ITEE[{\begin{Cor}[#1]}]{#1}{}{\begin{Cor}}\label{cor:#2}}{\end{Cor}}
\newenvironment{dfn}[2][]{\ITEE[{\begin{Def}[#1]}]{#1}{}{\begin{Def}}\label{def:#2}}{\end{Def}}
\newenvironment{exm}[2][]{\ITEE[{\begin{Exm}[#1]}]{#1}{}{\begin{Exm}}\label{exm:#2}}{\end{Exm}}
\newenvironment{lem}[2][]{\ITEE[{\begin{Lem}[#1]}]{#1}{}{\begin{Lem}}\label{lem:#2}}{\end{Lem}}
\newenvironment{pro}[2][]{\ITEE[{\begin{Pro}[#1]}]{#1}{}{\begin{Pro}}\label{pro:#2}}{\end{Pro}}
\newenvironment{rem}[2][]{\ITEE[{\begin{Rem}[#1]}]{#1}{}{\begin{Rem}}\label{rem:#2}}{\end{Rem}}
\newenvironment{thm}[2][]{\ITEE[{\begin{Thm}[#1]}]{#1}{}{\begin{Thm}}\label{thm:#2}}{\end{Thm}}

\newcommand{\COR}[2][!]{\ITEE{#1}{!}{Corollary~}\ITEE{#1}{s}{Corollaries~}\textup{\ref{cor:#2}}}
\newcommand{\DEF}[2][!]{\ITEE{#1}{!}{Definition~}\ITEE{#1}{s}{Definitions~}\textup{\ref{def:#2}}}
\newcommand{\LEM}[2][!]{\ITEE{#1}{!}{Lemma~}\ITEE{#1}{s}{Lemmas~}\textup{\ref{lem:#2}}}
\newcommand{\PRO}[2][!]{\ITEE{#1}{!}{Proposition~}\ITEE{#1}{s}{Propositions~}\textup{\ref{pro:#2}}}
\newcommand{\REM}[2][!]{\ITEE{#1}{!}{Remark~}\ITEE{#1}{s}{Remarks~}\textup{\ref{rem:#2}}}
\newcommand{\THM}[2][!]{\ITEE{#1}{!}{Theorem~}\ITEE{#1}{s}{Theorems~}\textup{\ref{thm:#2}}}

\newcommand{\KKK}{\mathbb{K}}
\newcommand{\NNN}{\mathbb{N}}
\newcommand{\RRR}{\mathbb{R}}
\newcommand{\MmM}{\EuScript{M}}
\newcommand{\Aa}{\mathfrak{A}}
\newcommand{\Bb}{\mathfrak{B}}
\newcommand{\Dd}{\mathfrak{D}}
\newcommand{\Mm}{\mathfrak{M}}
\newcommand{\Nn}{\mathfrak{N}}
\newcommand{\aaA}{\mathscr{A}}
\newcommand{\bbB}{\mathscr{B}}
\newcommand{\ffF}{\mathscr{F}}
\newcommand{\ggG}{\mathscr{G}}
\newcommand{\hhH}{\mathscr{H}}
\newcommand{\kkK}{\mathscr{K}}
\newcommand{\llL}{\mathscr{L}}
\newcommand{\mmM}{\mathscr{M}}

\newcommand{\dd}{\colon}
\newcommand{\df}{\stackrel{\textup{def}}{=}}
\newcommand{\dint}[1]{\,\textup{d} #1}
\newcommand{\epsi}{\varepsilon}
\newcommand{\geqsl}{\geqslant}
\newcommand{\leqsl}{\leqslant}
\newcommand{\scalar}[2]{\left\langle#1,#2\right\rangle}
\newcommand{\varempty}{\varnothing}

\newcommand{\RE}{\operatorname{Re}}

\newcommand{\tfcae}{the following conditions are equivalent:}
\newcommand{\iaoi}{if and only if}

\begin{document}

\title{Bounded convergence theorems}
\author[P. Niemiec]{Piotr Niemiec}
\address{Instytut Matematyki\\{}Wydzia\l{} Matematyki i~Informatyki\\{}
 Uniwersytet Jagiello\'{n}ski\\{}ul. \L{}ojasiewicza 6\\{}30-348 Krak\'{o}w\\{}Poland}
\email{piotr.niemiec@uj.edu.pl}
\begin{abstract}
There are presented certain results on extending continuous linear operators defined on spaces
of $E$-valued continuous functions (defined on a compact Hausdorff space $X$) to linear operators
defined on spaces of $E$-valued measurable functions in a way such that uniformly bounded sequences
of functions that converge pointwise in the weak (or norm) topology of $E$ are sent to sequences
that converge in the weak, norm or weak* topology of the target space. As an application, a new
description of uniform closures of convex subsets of $C(X,E)$ is given. Also new and strong results
on integral representations of continuous linear operators defined on $C(X,E)$ are presented. A new
classes of vector measures are introduced and various bounded convergence theorems for them are
proved.
\end{abstract}
\subjclass[2010]{Primary 46G10; Secondary 46E40.}
\keywords{Vector measure; dual Banach space; Riesz characterisation theorem; weakly sequentially
 complete Banach space; dominated convergence theorem; bounded convergence theorem; function space.}
\maketitle


\section{Introduction}

Lebesgue's dominated convergence theorem (for nonnegative measures) is a fundamental as well as
powerful tool which finds applications in many mathematical branches. (In this paper all measures
are meant to be countably additive.) Although nonnegative measures were naturally generalised
to vector-valued set functions (usually called \textit{vector measures}) many years ago (see, for
example, \cite{din}, \cite{d-u} or Chapter~IV in \cite{d-s}) and the above result waited many
generalisations, one of the disadvantages of vector integrals (of vector-valued functions with
respect to vector integrals) is the difficulty in verifying that a specific function is integrable.
For instance, if the total variation of a vector measure is infinite, not every bounded measurable
function with separable image is integrable, in the opposite to the scalar case (since every
scalar-valued measure automatically has finite variation). This causes that the concepts
of integrating vector-valued functions with respect to vector measures (proposed by Bartle
\cite{bar}, Dinculeanu \cite{din}, Goodrich \cite{go1,go2}, Lewis \cite{lew}, Tucker and Wayment
\cite{t-w}, Smith and Tucker \cite{s-t} and others) is not as popular as the classical theory
of measure and integration (and the theory of integrating vector-valued functions with respect
to nonnegative measures or scalar-valued functions with respect to vector measures; see, for
example, \cite{d-u}). In this paper we introduce a new class of vector measures with respect
to which all bounded measurable functions with separable images are integrable and for which
(strong) bounded convergence theorem holds (which may be seen as a counterpart of the Lebesgue
dominated convergence theorem). Our approach is based on results on extending continuous linear
operators (such as stated in the abstract). To formulate the main of them, let us first introduce
necessary definitions. Everywhere below $X$ and $\Omega$ are, respectively, a compact and a locally
compact Hausdorff space and $E$ and $F$ are Banach spaces.

\begin{dfn}{M(A)}
For a nonempty set $Z$, let $\ell_{\infty}(Z,E)$ stand for the Banach space of all $E$-valued
bounded functions on $Z$ (equipped with the sup-norm induced by the norm of $E$). For every set
$A \subset \ell_{\infty}(Z,E)$, the space \textit{$\mmM(A)$} is defined as the smallest set among
all $B \subset \ell_{\infty}(Z,E)$ such that:
\begin{enumerate}[(M1)]\addtocounter{enumi}{-1}
\item $A \subset B$;
\item whenever $f_n \in B$ are uniformly bounded and converge pointwise to $f \in
 \ell_{\infty}(Z,E)$ in the weak topology of $E$, then $f \in B$.
\end{enumerate}
It is an easy exercise that $\mmM(V)$ is a linear subspace of $\ell_{\infty}(Z,E)$ provided $V$ is
so.\par
By $C(X,E)$ ($C_0(\Omega,E)$) we denote the subspace of $\ell_{\infty}(X,E)$ (resp.
of $\ell_{\infty}(\Omega,E)$) consisting of all continuous functions from $X$ into $E$ (resp. from
$\Omega$ into $E$ that vanish at infinity). For simplicity, we put $\ell_{\infty}^{\KKK} \df
\ell_{\infty}(\NNN,\KKK)$ (where $\NNN \df \{1,2,\ldots\}$).
\end{dfn}

Our main result on extending continuous linear operators reads as follows.

\begin{thm}{1}
Let $V$ be a linear subspace of $C(X,E)$. Every continuous linear operator $T\dd V \to F^*$ is
uniquely extendable to a linear operator $\bar{T}\dd \mmM(V) \to F^*$ such that:
\begin{itemize}
\item[(BC*)] whenever $f_n \in \mmM(V)$ are uniformly bounded and converge pointwise to $f \in
 \mmM(V)$ in the weak topology of $E$, then $\bar{T} f_n$ converge to $\bar{T} f$ in the weak*
 topology of $F^*$.
\end{itemize}
Moreover, $\bar{T}$ is continuous and $\|\bar{T}\| = \|T\|$.
\end{thm}

In the above notation, ``BC'' is the abbreviation of \textit{bounded convergence} and ``*'' is
to emphasize that the \textit{final} convergence is in the weak* topology. In the sequel, we shall
continue this concept.\par
It is a matter of taste to think of integrals as derived from measures (a typical approach
in measure theory) or conversely (for example, starting from Riesz' characterisation theorem or from
the Daniell theory of integrals; see \cite{dan} or Chapter~XIII in \cite{mau}). In this paper
we follow the latter approach, generalising the classical Riesz characterisation theorem in a new
way, which led us to the introduction of a new class of vector measures:

\begin{dfn}{i-m}
For $T_n \in \llL(E,F)$ (where $\llL(E,F)$ stands for the Banach space of all continuous linear
operators from $E$ into $F$), the series $\sum_{n=1}^{\infty} T_n$ is said to be
\textit{independently} convergent if the series
\begin{equation}\label{eqn:i-m}
\sum_{n=1}^{\infty} T_n x_n
\end{equation}
is convergent in the norm topology of $F$ for every bounded sequence of elements $x_n$ of $E$.
(If this happens, the series \eqref{eqn:i-m} is unconditionally convergent.)\par
A set function $\mu\dd \Mm \to \llL(E,F)$ (where $\Mm$ is a $\sigma$-algebra of a set $Z$) is called
an \textit{i-measure} if $\mu(\bigcup_{n=1}^{\infty} A_n)x = \sum_{n=1}^{\infty} \mu(A_n)x$ (for
each $x \in E$) and the series $\sum_{n=1}^{\infty} \mu(A_n)$ is independently convergent for any
sequence of pairwise disjoint sets $A_n \in \Mm$. The \textit{total semivariation} $\|\mu\|_Z \in
[0,\infty]$ of $\mu$ is given by
\begin{multline}\label{eqn:semi}
\|\mu\|_Z \df \sup\Bigl\{\bigl\|\sum_{n=1}^N \mu(A_n) x_n\bigr\|\dd\ N < \infty,\ A_n \in \Mm
\textup{ are pairwise disjoint},\\x_n \in E,\ \|x_n\| \leqsl 1\Bigr\}
\end{multline}
(compare \S4 of Chapter~I in \cite{din}).
\end{dfn}

We shall prove in \LEM{op} that every independently convergent series of elements of $\llL(E,F)$ is
convergent in the norm topology of $\llL(E,F)$ (and thus every i-measure is a vector measure with
respect to the norm topology of $\llL(E,F)$). What is more, it turns out that each i-measure has
finite total semivariation (see \THM{fin}). This discovery enables us to define the vector integral
$\int_Z f \dint{\mu}$ of any $E$-valued bounded measurable function $f$ with separable image with
respect to a given $\llL(E,F)$-valued i-measure $\mu$ on a set $Z$. We also show that the operator
$\bar{T}$ given by $\bar{T} f \df \int_Z f \dint{\mu}$ satisfies condition (BC*) with the weak
topology of $F$ inserted in place of the weak* topology of $F^*$, or with the norm topologies on $E$
and $F$ (and $\mmM(V)$ replaced by the space of all functions $f$ with the properties specified
above). This is shown in \THM[s]{bwc} and \THM[]{bnc}. These remarks may justify a conclusion that
i-measures are the best counterparts (in the operator-valued case) of finite nonnegative
(or scalar-valued) measures.\par
Taking into account the Riesz characterisation theorem, continuous linear operators from Banach
spaces of the form $C(X,E)$ (into arbitrary Banach spaces) may be called abstract vector integrals.
There are a number of results which justify such a terminology (see, for example, \cite{go1,go2},
\cite{s-t} or Theorem~9 in \S5 of Chapter~III in \cite{din}). However, in most of them the final
vector measure is only finitely additive. In our characterisation (in a special case) the final
measure is an i-measure (and thus it is countably additive):

\begin{thm}{3}
Let $F$ be a weakly sequentially complete Banach space or a dual Banach space containing
no isomorphic copy of $\ell_{\infty}^{\RRR}$ and let $\Omega$ be a locally compact Hausdorff space.
For every continuous linear operator $T\dd C_0(\Omega,E) \to F$ there exists a unique regular Borel
i-measure $\mu\dd \Bb(\Omega) \to \llL(E,F)$ such that
\begin{equation}\label{eqn:vint}
T f = \int_{\Omega} f \dint{\mu} \qquad (f \in C_0(\Omega,E)).
\end{equation}
Conversely, if $\mu\dd \Bb(\Omega) \to \llL(E,F)$ is an arbitrary regular i-measure \textup{(}and
$F$ is an arbitrary Banach space\textup{)}, then \eqref{eqn:vint} correctly defines a continuous
linear operator $T\dd C_0(\Omega,E) \to F$ such that $\|T\| = \|\mu\|_{\Omega}$.
\end{thm}

We also give an integral representation of continuous linear operators from $C_0(\Omega,E)$ which
take values in arbitrary Banach spaces. This is done with the help of so-called \textit{weak*}
i-measures, introduced and discussed in Section~6.\par
As a consequence of \THM{3} and bounded convergence theorems for i-measures, we obtain a new result
on the description of the uniform closure of a convex subset of $C(X,E)$:

\begin{thm}{2}
In each of the three cases specified below, the norm closure of a convex subset $\kkK$
of $C_0(\Omega,E)$ coincides with the set of all functions $f \in C_0(\Omega,E)$ such that
$f\bigr|_L \in \mmM\bigl(\kkK\bigr|_L\bigr)$ \textup{(}where $\kkK\bigr|_L \df \{g\bigr|_L \in
C(L,E)\dd\ g \in \kkK\}$\textup{)} for any compact set $L \subset \Omega$:
\begin{itemize}
\item $\Omega$ is compact; or
\item $\kkK$ is bounded; or
\item $E$ is a $C^*$-algebra and $\kkK$ is a $*$-subalgebra of $C_0(\Omega,E)$.
\end{itemize}
\end{thm}

The above result seems to be a convenient tool. Recently we use some of its variations to describe
models for subhomogeneous $C^*$-algebras (which may be seen as a solution of a long-standing
problem). The paper on this is in preparation.\par
The paper is organised as follows. Section~2 is devoted to the proof of \THM{1} and some of its
generalisations. In Section~3 we introduce \textit{variationally} sequentially complete Banach
spaces (which all weakly sequentially complete as well as all dual Banach spaces belong to), give
a new characterisation of weakly sequentially complete Banach spaces and formulate a variation
of \THM{1} for operators taking values in variationally sequentially complete Banach spaces.
The fourth part discusses in details i-measures and contains a preliminary material to the proof
of \THM{3}. Section~5 is devoted to weak* i-measures. Section~6 discusses regular i-measures as well
as regular weak* i-measures. It contains a proof of \THM{3} and its variations for operators taking
values in variationally sequentially complete Banach spaces containing no isomorphic copy
of $\ell_{\infty}^{\RRR}$ (see \THM{vsc}) and in dual Banach spaces (consult \THM{W*}) as well as
totally arbitrary Banach spaces (see \COR{Riesz}). The last, seventh part is devoted to the proof
of \THM{2} and some of its variations. We give there also an illustrative application and an example
showing that the boundedness condition in the second case of \THM{2} cannot be, in general, dropped.

\subsection*{Notation and terminology} Throughout the whole paper, all topological spaces are
assumed to be Hausdorff. $X$, $\Omega$, and $E$ and $F$ are reserved to denote, respectively,
a compact space, a locally compact space and two Banach spaces over the field $\KKK$ of real
or complex numbers. The dual of a locally convex topological vector space $(G,\tau)$ is denoted
by $(G,\tau)^*$ (or simply $G^*$ if it is known from the context with respect to which topology
on $G$ the dual is taken) and is understood as the vector space of all continuous linear functionals
on $(G,\tau)$. A subset $A$ of a topological space $Y$ is \textit{sequentially closed} if $A$
contains the limits of all convergent (in $Y$) sequences whose entries belong to $A$. $A$ is
\textit{$\sigma$-compact} if it is a countable union of compact subsets of $Y$. Finally, $\Bb(Y)$
stands for the $\sigma$-algebra of all \textit{Borel} sets in $Y$; that is, $\Bb(Y)$ is the smallest
$\sigma$-algebra of subsets of $Y$ that contains all open sets.\par
All notations and terminologies introduced in \DEF[s]{M(A)} and \DEF[]{i-m} are obligatory.

\section{Extending linear operators}

\begin{dfn}{measurable}
Let $\Mm$ be a $\sigma$-algebra on a set $Z$. A function $f\dd Z \to E$ is said to be
\textit{$\Mm$-measurable} if
\begin{itemize}
\item $f(X)$ is a separable subspace of $E$; and
\item $f$ is weakly $\Mm$-measurable; that is, for any $\psi \in E^*$, the function $\psi \circ f\dd
 Z \to \KKK$ is $\Mm$-measurable.
\end{itemize}
Thanks to a theorem of Pettis \cite{pet}, $f$ is $\Mm$-measurable iff $f(Z)$ is a separable subspace
of $E$ and the inverse image of every Borel set in $E$ under $f$ belongs to $\Mm$.\par
$M_{\Mm}(Z,E)$ is defined as the subspace of $\ell_{\infty}(Z,E)$ consisting of all bounded
$\Mm$-measurable functions $f\dd Z \to E$.\par
For a compact space $X$, let $\Mm(X)$ be the smallest $\sigma$-algebra on $X$ that contains all
closed sets in $X$ of type $\ggG_{\delta}$. \textit{$M(X,E)$} stands for $M_{\Mm(X)}(X,E)$.
\end{dfn}

It is worth noting here that, in general, not every open set in $X$ belongs to $\Mm(X)$. But if $X$
is metrisable (or, more generally, perfectly normal), then $\Mm(X) = \Bb(X)$.\par
The next result is certainly known. For the reader's convenience, we give its proof.

\begin{lem}{1}
$\mmM(C(X,E)) = M(X,E)$.
\end{lem}
\begin{proof}
First of all, observe that $\Mm(X)$ is the smallest $\sigma$-algebra on $X$ with respect to which
all $\KKK$-valued continuous functions on $X$ are measurable. It is therefore an elementary exercise
to check that the set $B = M(X,E)$ satisfies conditions (M0)--(M1) for $A = C(X,E)$. Consequently,
$\mmM(C(X,E)) \subset M(X,E)$. Instead of proving the reverse inclusion, we shall show a little bit
more: that $M(X,E)$ coincides with the smallest set $N(E)$ among all $B \subset \ell_{\infty}(X,E)$
which include $C(X,E)$ and satisfy the condition:
\begin{enumerate}[(M1')]
\item whenever $f_n \in B$ are uniformly bounded and converge pointwise to $f \in
 \ell_{\infty}(X,E)$ in the norm topology of $E$, then $f \in B$.
\end{enumerate}
To this end, for any $A \subset X$, denote by $j_A\dd X \to \{0,1\}$ the characteristic function
of $A$. First we assume $E = \KKK$. Observe that $N(\KKK)$ is a unital subalgebra
of $\ell_{\infty}(X,\KKK)$. This implies that $\Nn \df \{A \in \Mm(X)\dd\ j_A \in N(\KKK)\}$ is
a $\sigma$-algebra on $X$. So, to conclude that $\Nn = \Mm(X)$, it suffices to show that each closed
set of type $\ggG_{\delta}$ belongs to $\Nn$. But this is immediate, since for any such set $K$
there are sequences $U_1 \supset U_2 \supset \ldots$ of open sets in $X$ and $f_1,f_2,\ldots\dd X
\to [0,1]$ of continuous functions such that $j_K \leqsl f_n \leqsl j_{U_n}$ and $K =
\bigcap_{n=1}^{\infty} U_n$. Consequently, $j_K$ is the pointwise limit of $f_n$'s and hence $K \in
\Nn$. This shows that $\Nn = \Mm(X)$. Now, since every scalar-valued bounded $\Mm(X)$-measurable
function is a uniform limit of linear combinations of characteristic functions of members
of $\Mm(X)$, we get that $M(X,\KKK) \subset N(\KKK)$. We turn to the general case.\par
For simplicity, we shall call any function $u\dd X \to E$ such that $u(X)$ is countable (finite
or not) and the inverse image of every point of $E$ under $u$ is a member of $\Mm(X)$
\textit{semisimple}. For any scalar-valued function $f\dd X \to \KKK$ and each vector $x \in E$,
we use $f(\cdot)x$ to denote a function from $X$ into $E$, computed pointwise. Now fix $e \in E$ and
consider families $F(e) \df \{u \in M(X,\KKK)\dd\ u(\cdot)e \in N(E)\}$ and $\Mm_e \df \{B \in
\Mm(X)\dd\ j_B \in F(e)\}$. Since $C(X,\KKK) \subset F(e)$, it follows from the previous part
of the proof that $F(e) = M(X,\KKK)$ and $\Mm(e) = \Mm(X)$. One easily deduces from these
connections and (M1') that
\begin{itemize}
\item[($\star$)] any semisimple function $u\dd X \to E$ belongs to $N(E)$.
\end{itemize}
Now take any $u \in M(X,E)$. Since the range of $u$ is a separable space and $u$ is weakly
$\Mm(X)$-measurable, one concludes that:
\begin{itemize}
\item the inverse image of any closed ball in $E$ under $u$ belongs to $\Mm(X)$;
\item for any $\epsi > 0$, there exists a countable (finite or not) collection of pairwise disjoint
 members of $\Mm(X)$ whose union coincides with $X$ and images under $u$ are contained in closed
 $\epsi$-balls of $E$.
\end{itemize}
Now using the latter of the above properties, for each $n > 0$, construct a semisimple function
$u_n\dd X \to E$ whose uniform distance from $u$ is less than $1/n$. So, $u$ is a uniform limit
of semisimple function and hence $u \in N(E)$, by ($\star$).
\end{proof}

Although the next lemma is very simple, it is crucial for our further purposes.

\begin{lem}{2}
Let $Y$ be a compact space and $U\dd E \to C(Y,\KKK)$ be a linear isometric embedding. For each
$v \in \ell_{\infty}(X,E)$ let $L v\dd X \times Y \to \KKK$ be given by $(L v)(x,y) \df U(v(x))(y)$.
Then the assignment $v \mapsto L v$ defines a linear isometric embedding $L$ of $\ell_{\infty}(X,E)$
into $\ell_{\infty}(X \times Y,\KKK)$ such that:
\begin{enumerate}[\upshape(L1)]
\item $L(C(X,E)) \subset C(X \times Y,\KKK)$;
\item if $v_n \in \ell_{\infty}(X,E)$ are uniformly bounded and converge pointwise
 to $v \in \ell_{\infty}(X,E)$ in the weak topology of $E$, then $L v_n$ are uniformly bounded
 as well and converge pointwise to $L v$;
\item for any set $A \subset \ell_{\infty}(X,E)$, $L(\mmM(A)) \subset \mmM(L(A))$ where the sets
 $\mmM(A)$ and $\mmM(L(A))$ are computed in, respectively, $\ell_{\infty}(X,E)$ and
 $\ell_{\infty}(X \times Y,\KKK)$.
\end{enumerate}
\end{lem}
\begin{proof}
It is readily seen that $L\dd \ell_{\infty}(X,E) \to \ell_{\infty}(X \times Y,\KKK)$ is linear and
isometric. Point (L1) is a well-known topological result---consult, for example, Theorems~3.4.7,
3.4.8 and 3.4.9 in \cite{eng}. (L2) follows from the facts that $U$ is continuous in the weak
topologies of $E$ and $C(Y,\KKK)$, and the weak topology of $C(Y,\KKK)$ is finer than the pointwise
convergence topology. Finally, (L3) is implied by (L2).
\end{proof}

Let us call a locally convex topological vector space $G$ \textit{initial} if its topology coincides
with the weak topology of $G$. Equivalently, $G(,\tau_0)$ is initial iff $\tau_0$ is the coarsest
topology among all locally convex topologies $\tau$ on $G$ for which the sets $(G,\tau)^*$ and
$(G,\tau_0)^*$ (considered here with no topology) coincide. Important examples of such spaces are
Banach spaces equipped with the weak topologies as well as dual Banach spaces equipped with
the weak* topologies. Recall that $G$ is \textit{sequentially complete} if every Cauchy sequence
in $G$ is convergent. The following result is a generalisation of \THM{1}:

\begin{thm}{1'}
Let $G$ be an initial sequentially complete locally convex topological vector space and $V$ be
a linear subspace of $C(X,E)$. Every continuous linear operator $T\dd V \to G$ is uniquely
extendable to a linear operator $\bar{T}\dd \mmM(V) \to G$ such that:
\begin{itemize}
\item[(BC')] whenever $f_n \in \mmM(V)$ are uniformly bounded and converge pointwise to $f \in
 \mmM(V)$ in the weak topology of $E$, then $\bar{T} f_n$ converge to $\bar{T} f$.
\end{itemize}
Moreover, $\bar{T}$ is continuous.
\end{thm}
\begin{proof}
It follows from (BC') and the very definition of $\mmM(V)$ that $\bar{T}$ is unique. To establish
the existence of $\bar{T}$, first note that the initiality and sequential completeness of $G$ imply
that:
\begin{itemize}
\item[(CC)] if $z_n \in G$ are such that $\psi(z_n)$ converge (in $\KKK$) for any $\psi \in G^*$,
 then $z_n$ converge (in $G$).
\end{itemize}
Next, there is an isometric linear embedding $U\dd E \to C(Y,\KKK)$ for a suitably chosen compact
space $Y$. Let $L\dd \ell_{\infty}(X,E) \to \ell_{\infty}(X \times Y,\KKK)$ be as specified
in \LEM{2}. We put $W \df L(V)$ and define $S\dd W \to G$ by $S \df T \circ (L\bigr|_V)^{-1}$.
It is enough to show that there is a linear extension $\bar{S}\dd \mmM(W) \to G$ of $S$ (where
$\mmM(W)$ is computed in $\ell_{\infty}(X \times Y,\KKK)$) such that:
\begin{itemize}
\item[(BC'')] whenever $f_n \in \mmM(W)$ are uniformly bounded and converge pointwise to $f \in
 \mmM(W)$, then $\bar{S}(f_n)$ converge to $\bar{S}(f)$,
\end{itemize}
because then $\bar{T} \df \bar{S} \circ L\bigr|_{\mmM(V)}$ is well defined (by condition (L3)
of \LEM{2}), extends $T$ and satisfies (BC') (thanks to (L2)). For simplicity, everywhere below
$\alpha$ denotes an arbitrary countable ordinal. To establish the existence of $\bar{S}$, for any
$\alpha$, we define a space $W_{\alpha}$ by transfinite induction as follows: $W_0 = W$ and for
$\alpha > 0$, $W_{\alpha}$ consists of all pointwise limits of uniformly bounded sequences from
$\bigcup_{\xi<\alpha} W_{\xi}$ (convergent in the pointwise topology). It is easy to check that each
of $W_{\alpha}$ is a linear subspace of $C(X \times Y,\KKK)$ and that $\mmM(W) = \bigcup_{\alpha}
W_{\alpha}$. Since any sequence of members of $\mmM(W)$ is contained in $W_{\alpha}$ for some
$\alpha$, it suffices to show that there exists a transfinite sequence $S_{\alpha}\dd W_{\alpha} \to
G$ of linear operators such that:
\begin{enumerate}[(E1)]
\item $S_0 = S$;
\item $S_{\alpha}$ extends $S_{\xi}$ provided $\xi < \alpha$;
\item whenever $f_n \in W_{\alpha}$ are uniformly bounded and converge pointwise to $f \in
 W_{\alpha}$, then $S_{\alpha} f_n$ converge to $S_{\alpha} f$
\end{enumerate}
(because then $\bar{S}$ may simply be defined by $\bar{S} f \df S_{\alpha} f$ where $\alpha$ is
chosen so that $f \in W_{\alpha}$). It follows from the Hahn-Banach and the Riesz characterisation
theorems that for any $\psi \in G^*$, there is a $\KKK$-valued regular Borel measure $\mu_{\psi}$
on $X \times Y$ such that:
\begin{equation*}
\psi(S f) = \int_{X \times Y} f \dint{\mu_{\psi}} \qquad (f \in W).
\end{equation*}
Define $S_0$ as specified in (E1) and assume that for some $\alpha > 0$, $S_{\xi}$ is defined for
any $\xi < \alpha$ in a way such that for each $\psi \in G^*$,
\begin{equation}\label{eqn:repr}
\psi(S_{\xi} f) = \int_{X \times Y} f \dint{\mu_{\psi}} \qquad (f \in W_{\xi}).
\end{equation}
We shall define $S_{\alpha}$ so that \eqref{eqn:repr} holds for $\xi = \alpha$ and then we shall
check that conditions (E2)--(E3) are satisfied. Let $u \in W_{\alpha}$. There is a uniformly bounded
sequence $u_n \in W_{\xi_n}$ (with $\xi_n < \alpha$) which converges pointwise to $u$. It then
follows from Lebesgue's dominated convergence theorem and \eqref{eqn:repr} that
\begin{equation}\label{eqn:seq}
\lim_{n\to\infty} \psi(S_{\xi_n} u_n) = \int_{X \times Y} u \dint{\mu_{\psi}}
\end{equation}
for each $\psi \in G^*$. So, we conclude from (CC) that $S_{\xi_n} u_n$ converge. We define
$S_{\alpha} u$ as the limit of the last mentioned sequence. It follows from \eqref{eqn:seq} that
\eqref{eqn:repr} is satisfied for $\xi = \alpha$ and $f = u$ (and any $\psi \in G^*$). This implies
that the definition of $S_{\alpha} u$ is independent of the choice of the functions $u_n$. Finally,
\eqref{eqn:repr} applied for all $\xi \leqsl \alpha$ shows that (E2) holds, and combined with
Lebesgue's dominated convergence theorem gives (E3) (because $G$ is initial).\par
To complete the proof, it remains to observe that the continuity of $\bar{T}$ follows from (BC')
(since $\mmM(V)$ is metrisable, it suffices to check the sequential continuity).
\end{proof}

\begin{proof}[Proof of \THM{1}]
Taking into account \THM{1'}, it is enough to verify that $F^*$ is initial and sequentially
complete in the weak* topology, and that the extension of $T$ does not increase the norm. Both
the above properties of $F^*$ are immediate. And to convince oneself that $\|\bar{T}\| = \|T\|$,
it suffices to repeat the proof of \THM{1'} and check that $\|S_{\alpha}\| = \|S\|$ for each
countable ordinal $\alpha$, which may simply be provided by choosing the measures $\mu_{\psi}$
(for $\psi \in F = (F,\textup{weak*})^*$) appearing in \eqref{eqn:repr} so that the total variation
$|\mu_{\psi}|(X \times Y)$ of $\mu_{\psi}$ does not exceed $\|S\| \cdot \|\psi\|$.
\end{proof}

As an immediate consequence of \THM{1} and \LEM{1}, we obtain the following result, announced
in the abstract.

\begin{cor}{C-M}
Every continuous linear operator $T\dd C(X,E) \to F^*$ is uniquely extendable to a linear operator
$\bar{T}\dd M(X,E) \to F^*$ satisfying condition \textup{(BC*)} of \THM{1} with $M(X,E)$ inserted
in place of $\mmM(V)$. Moreover, $\bar{T}$ is continuous and $\|\bar{T}\| = \|T\|$.
\end{cor}

\section{Variational sequential completeness}

Recall that a Banach space is \textit{weakly sequentially complete} (briefly, \textit{wsc}) if it is
sequentially complete with respect to the weak topology. Each reflexive Banach space is wsc and
$\ell_1$ is an example of a nonreflexive wsc Banach space. These two exclusive examples are,
in a sense, exhaustive. Namely, by a celebrated result due to Rosenthal \cite{ros}, every wsc Banach
space is either reflexive or contains an isomorphic copy of $\ell_1$. An interesting
characterisation of wsc Banach spaces is given below.

\begin{pro}{wsc}
For a Banach space $F$ \tfcae
\begin{enumerate}[\upshape(a)]
\item Every continuous linear operator $T\dd V \to F$ from a linear subspace $V$ of \textup{(}some
 Banach space of the form\textup{)} $C(X,E)$ extends uniquely to a linear operator $\bar{T}\dd
 \mmM(V) \to F$ such that:
 \begin{itemize}
 \item[(BC)] whenever $f_n \in \mmM(V)$ are uniformly bounded and converge pointwise to $f \in
  \mmM(V)$ in the weak topology of $E$, then $\bar{T} f_n$ converge to $\bar{T} f$ in the weak
  topology of $F$.
 \end{itemize}
 \textup{(}Moreover, $\bar{T}$ is continuous and $\|\bar{T}\| = \|T\|$.\textup{)}
\item $F$ is wsc.
\end{enumerate}
\end{pro}
\begin{proof}
One easily deduces from \THM{1'} that (a) is implied by (b). (The additional claim of (a) may be
shown as explained in the proof of \THM{1}.) To see that the reverse implication also holds, take
a sequence $z_1,z_2,\ldots \in F$ which is Cauchy in the weak topology. Define $X$ as the closed
unit ball of $F^*$ equipped with the weak* topology and put $E \df \KKK$. Further, for each $x \in
F$, we use $e_x\dd X \to E$ to denote the evaluation map at $x$; that is, $e_x(\psi) = \psi(x)$.
Denote by $F_0$ the linear span of all $z_n$, put $V \df \{e_z\dd\ z \in F_0\} \subset C(X,E)$ and
define $T\dd V \to F$ by $T e_z \df z$. It is readily seen that $T$ is continuous (even isometric)
and linear. So, it follows from (a) that there is a linear extension $\bar{T}\dd \mmM(V) \to F$
of $T$ which satisfies (BC). Since the sequence of all $z_n$ is Cauchy in the weak topology of $F$,
the formula $u(\psi) \df \lim_{n\to\infty} \psi(z_n)$ correctly defines a function $u\dd X \to E$.
Notice that the functions $e_{z_n}$ are uniformly bounded and converge pointwise to $u$. Thus,
$u \in \mmM(V)$ and, by (BC), $z_n = \bar{T} e_{z_n}$ converge to $\bar{T} z$ in the weak topology
of $F$.
\end{proof}

\THM[s]{1} and \THM[]{1'} and \PRO{wsc} suggest to distinguish certain Banach spaces, which we do
below.

\begin{dfn}{vsc}
A Banach space $F$ is said to be \textit{variationally sequentially complete} (briefly,
\textit{vsc}) if there is a set $\ffF \subset F^*$ such that:
\begin{enumerate}[(vsc1)]
\item there is a positive constant $\lambda$ such that for any $x \in F$,
 \begin{equation*}
 \frac{1}{\lambda} \sup\{|\psi(x)|\dd\ \psi \in \ffF\} \leqsl \|x\| \leqsl \lambda
 \sup\{|\psi(x)|\dd\ \psi \in \ffF\};
 \end{equation*}
\item whenever $z_n \in F$ are uniformly bounded and $\psi(z_n)$ converge for each $\psi \in \ffF$,
 then there exists $z \in F$ such that $\lim_{n\to\infty} \psi(z_n) = \psi(z)$ for all $\psi \in
 \ffF$.
\end{enumerate}
It is worth noting that the point $z$ appearing in (vsc2) is unique. For simplicity, we shall denote
it by $\ffF$-$\lim_{n\to\infty} z_n$.\par
More specifically, $F$ is called \textit{$\alpha$-vsc} (where $\alpha \geqsl 1$) if there exists
$\ffF \subset F^*$ such that (vsc1)--(vsc2) hold with $\lambda = \alpha$.
\end{dfn}

Basic examples of vsc spaces are wsc as well as dual Banach spaces. It is also clear that a Banach
space is vsc provided it is isomorphic to a vsc Banach space.\par
It is an easy exercise to show that a Banach space is wsc iff it is sequentially closed in the weak*
topology of its second dual. A counterpart of this characterisation for vsc Banach spaces is given
below.

\begin{pro}{vsc-dual}
A Banach space $F$ is vsc iff it is isomorphic to a linear subspace $W$ of some dual Banach space
$Z^*$ such that $W$ is sequentially closed in the weak* topology of $Z^*$.
\end{pro}
\begin{proof}
First assume $F$ is vsc and let $\ffF \subset F^*$ be such that (vsc1)--(vsc2) are fulfilled. We put
$Z \df \ell_1(\ffF,\KKK)$; that is, $Z$ consists of all functions $u\dd \ffF \to \KKK$ such that
$\|u\| \df \sum_{\psi \in \ffF} |u(\psi)| < \infty$. Then $Z^* = \ell_{\infty}(\ffF,\KKK)$. Define
$\Phi\dd F \to \ell_{\infty}(\ffF,\KKK)$ by $(\Phi f)(\psi) = \psi(f)$. It follows from (vsc1) that
$\Phi$ is a well defined topological embedding. We claim that $W \df \Phi(F)$ is sequentially closed
in the weak* topology of $\ell_{\infty}(\ffF,\KKK)$. To see this, let $z_n \in F$ be such that
$\Phi(z_n)$ converge to $u \in \ell_{\infty}(\ffF,\KKK)$ in the weak* topology. Then $\Phi(z_n)$ are
uniformly bounded and, consequently, so are $z_n$. Furthermore, $\psi(z_n)$ converge for any $\psi
\in \ffF$. So, (vsc2) implies that $z \df \ffF\textup{-}\lim_{n\to\infty} z_n$ well defines a vector
in $F$ and $u = \Phi(z)$.\par
Conversely, assume $F$ is isomorphic to $W$ where $W \subset Z^*$ is as specified
in the proposition. It suffices to check that $W$ is vsc. For any $x \in Z$, let $j_x \in W^*$ be
given by $j_x(\psi) \df \psi(x)$. Put $\ffF \df \{j_x\dd\ x \in Z,\ \|x\| \leqsl 1\}$. We see that
(vsc1) holds with $\lambda = 1$. Now assume $\varphi_n \in W$ are such that $\psi(\varphi_n)$
converge for any $\psi \in \ffF$. Then $\varphi_n$ converge pointwise (on the whole $Z$) to some
function $\varphi\dd Z \to \KKK$. It now follows from the Uniform Boundedness Principle that
$\varphi \in Z^*$ and, consequently (since $W$ is sequentially closed), $\varphi \in W$. This shows
that (vsc2) holds and we are done.
\end{proof}

As a consequence, we obtain

\begin{pro}{vsc}
Every continuous linear operator $T\dd V \to F$ from a linear subspace $V$ of \textup{(}some space
of the form\textup{)} $C(X,E)$ into a vsc Banach space $F$ is extendable to a continuous linear
operator $\bar{T}\dd \mmM(V) \to F$.
\end{pro}
\begin{proof}
Let $\Phi\dd F \to W$ be an isomorphism where $W$ is a linear subspace of a dual Banach space $Z^*$
that is sequentially closed in the weak* topology (see \PRO{vsc-dual}). Put $L \df \Phi \circ T\dd V
\to W \subset Z^*$. It follows from \THM{1} that there exists a linear extension $\bar{L}\dd \mmM(V)
\to F^*$ of $L$ such that $\|\bar{L}\| = \|L\|$. What is more, the proof of \THM{1'} shows that all
values of $\bar{L}$ belong to $W$, since $W$ is sequentially closed in $Z^*$. Thus $\bar{T} \df
\Phi^{-1} \circ \bar{L}$ well defines a continuous linear extension of $T$ we searched for.
\end{proof}

For $V = C(X,E)$ (and under an additional assumption on $F$), \PRO{vsc} shall be strengthened
in \COR{vsc}.

\begin{rem}{vsc}
The above proof shows that, under the notation of \PRO{vsc}:
\begin{itemize}
\item every continuous linear operator $T\dd V \to F$ extends to a continuous linear operator
 $\bar{T}\dd \mmM(V) \to F$ such that $\|\bar{T}\| \leqsl \lambda^2 \|T\|$ provided $F$ is
 $\lambda$-vsc;
\item a linear subspace of a dual Banach space which is sequentially closed in the weak* topology is
 $1$-vsc.
\end{itemize}
We shall use these observations in the sequel.
\end{rem}

\PRO{vsc} combined with \REM{vsc} yields

\begin{cor}{w*sc}
Let $F_{sc}$ be the smallest linear subspace of $F^{**}$ that contains $F$ and is sequentially
closed in the weak* topology of $F^{**}$. Every continuous linear operator $T\dd V \to F$ from
a linear subspace $V$ of \textup{(}some space of the form\textup{)} $C(X,E)$ is extendable
to a continuous linear operator $\bar{T}\dd \mmM(V) \to F_{sc}$ such that $\|\bar{T}\| = \|T\|$.
\end{cor}

\begin{exm}{1vsc}
Let $V$ be a linear subspace of $C(X,F)$ where $F$ is a reflexive Banach space. Then $\mmM(V)$ is
a $1$-vsc (in particular, $M(X,F)$ is a $1$-vsc). Indeed, $\ell_{\infty}(X,F)$ is the dual Banach
space of
\begin{equation*}
\ell_1(X,F^*) \df \Bigl\{u\dd X \to F^*|\quad (\|u\| \df) \sum_{x \in X} \|u(x)\| < \infty\Bigr\}
\end{equation*}
and a sequence of elements of $\ell_{\infty}(X,F)$ converges in the weak* topology iff it is
uniformly bounded and converges (to the same limit) pointwise in the weak topology of $F$ (because
$F$ is reflexive). We conclude that $\mmM(V)$ is sequentially closed in the weak* topology
of $\ell_{\infty}(X,F)$. So, the assertion follows from \REM{vsc}.\par
The same argument proves that $M_{\Mm}(Z,E)$ is $1$-vsc provided $E$ is reflexive and $\Mm$ is
a $\sigma$-algebra on $Z$.
\end{exm}

In the last section we shall prove a counterpart of \THM{3} for vsc Banach spaces $F$ which contain
no isomorphic copy of $\ell_{\infty}^{\RRR}$ (see \THM{vsc}). It seems to be interesting and helpful
to know more about vsc Banach spaces. This will be the subject of our further studies.

\section{Strong results on vector integrals}

As we mentioned in the introductory part, taking into account the Riesz characterisation theorem,
continuous linear operators from $C(X,E)$ into arbitrary Banach spaces may be called (abstract)
\textit{vector integrals}. Such a terminology may be justified, for example, by a theorem formulated
below.

\begin{thm}[Theorem~9 in \S5 of Chapter~III in \cite{din}]{din}
For every continuous linear operator $T\dd C(X,E) \to F$ and a closed linear norming subspace $Z$
of $F^*$, there exists a finitely additive set function $m\dd \Bb(X) \to \llL(E,Z^*)$ such that
\begin{equation}\label{eqn:int}
Tf = \int_X f \dint{\mu} \qquad (f \in C(X,E)).
\end{equation}
\end{thm}

For a proof and the definition of the integral appearing in \eqref{eqn:int}, consult \cite{din}.
Other results in this fashion may be found, for example, in \cite{go1,go2} and \cite{lew}.\par
The reader should notice that, under the notation of \THM{din}, $Z^*$ differs from $F$, unless $F$
is a dual Banach space. \THM{3} shows that in the case when $F$ is wsc, the set function $\mu$ may
always be taken so that it takes values in $\llL(E,F)$. (More generally, it suffices that $F$ is vsc
and contains no isomorphic copy of $\ell_{\infty}^{\RRR}$; see \THM{vsc} in the last section.)
$L^1([0,1])$ is an example of a wsc Banach space which is isomorphic to no dual Banach space.
To formulate our first result on vector measures, we recall

\begin{dfn}{op-vec}
Whenever $\Mm$ is a $\sigma$-algebra of subsets of some set, a set function $\mu\dd \Mm \to
\llL(E,F)$ is said to be an \textit{operator measure} if for any $x \in E$ and $\psi \in F^*$,
the set function $\Mm \ni A \mapsto \psi(\mu(A) x) \in \KKK$ is a scalar-valued measure. According
to the Orlicz-Pettis theorem (see, for example, Corollary~4 on page~22 in \cite{d-u}), if $\mu$ is
an operator-valued measure and $A_n \in \Mm$ are pairwise disjoint, then $\mu(\bigcup_{n=1}^{\infty}
A_n) x = \sum_{n=1}^{\infty} \mu(A_n) x$ (the convergence in the norm topology) for each $x \in
E$.\par
Similarly, a set function $\mu\dd \Mm \to F$ is said to be a \textit{vector measure} if for any
$\psi \in F^*$, the set function $\Mm \ni A \mapsto \psi(\mu(A)) \in \KKK$ is a scalar-valued
measure. Equivalently, $\mu$ is a vector measure iff $\mu(\bigcup_{n=1}^{\infty} A_n) =
\sum_{n=1}^{\infty} \mu(A_n)$ (the convergence in the norm topology) for any sequence of pairwise
disjoint sets $A_n \in \Mm$.\par
Finally, a set function $\mu\dd \Mm \to F^*$ is said to be a \textit{weak* vector measure}
if the set function $\Mm \ni A \mapsto (\mu(A))(f) \in \KKK$ is a (scalar-valued countably additive)
measure for any $f \in F$.\par
It is worth emphasizing here that a set function $\mu\dd \Mm \to \llL(E,F)$ is an operator measure
provided it is a vector measure, but the reverse implication may fail to hold.
\end{dfn}

The reader is referred to \DEF{i-m} (in the introductory section) to recall the notion
of an i-measure. The next result shows that every such a set function is a vector measure.

\begin{lem}{op}
A series $\sum_{n=1}^{\infty} T_n$ with summands in $\llL(E,F)$ is convergent in the norm topology
of $\llL(E,F)$ provided it is independently convergent. In particular, every i-measure is a vector
measure.
\end{lem}
\begin{proof}
By the assumptions, for each $n > 0$, the formula $S_n x \df \sum_{k=n}^{\infty} T_k x$ correctly
defines a linear operator $S_n\dd E \to F$. It follows from the Uniform Boundedness Principle that
$S_n \in \llL(E,F)$. It remains to check that $\lim_{n\to\infty} \|S_n\| = 0$. We assume,
on the contrary, that $\|S_n\| > \epsi$ for some $\epsi > 0$ and infinitely many $n$. We shall mimic
the proof of Schur's lemma (on weakly convergent sequences in $\ell_1$). Let $\nu_1$ and $x_1 \in E$
be, respectively, a positive integer and a unit vector such that $\|S_{\nu_1} x_1\| > \epsi$.
It follows from our hypothesis that there is $\nu_2 > \nu_1$ such that $\|\sum_{k=\nu_1}^{\nu_2-1}
T_k x_1\| > \epsi$ and $\|S_{\nu_2}\| > \epsi$. We continue this procedure: if $\nu_1 < \ldots <
\nu_m$ are integers (where $m > 1$) and $x_1,\ldots,x_{m-1}$ are unit vectors of $E$ such that
$\|S_{\nu_m}\| > \epsi$ and
\begin{equation}\label{eqn:x}
\Bigl\|\sum_{k=\nu_j}^{\nu_{j+1}-1} T_k x_j\Bigr\| > \epsi
\end{equation}
for each $j \in \{1,\ldots,m-1\}$, we may find an integer $\nu_{m+1} > \nu_m$ and a unit vector $x_m
\in E$ for which $\|S_{\nu_{m+1}}\| > \epsi$ and \eqref{eqn:x} holds for $j = m$. In this way
we obtain a bounded sequence of vectors $x_n$ and an increasing sequence of integers $\nu_n$ such
that \eqref{eqn:x} holds for each $j$. But, if follows from the assumptions of the lemma that
the series $\sum_{n=1}^{\infty} (\sum_{k=\nu_n}^{\nu_{n+1}-1} T_k x_n)$ converges in the norm
topology, which contradicts \eqref{eqn:x}.\par
The additional claim of the lemma simply follows.
\end{proof}

Another strong property of i-measures is established below.

\begin{thm}{fin}
Every i-measure has finite total semivariation.
\end{thm}
\begin{proof}
Let $\mu\dd \Mm \to \llL(E,F)$ be an i-measure defined on a $\sigma$-algebra $\Mm$ of subsets
of a set $Z$. Suppose, on the contrary, that $\|\mu\|_Z = \infty$. For an arbitrary set $A \in \Mm$
we may similarly define $\|\mu\|_A \in [0,\infty]$ as the supremum of all numbers of the form
$\|\sum_{n=1}^N \mu(A_n) x_n\|$ where $N$ is finite, $A_n \in \Mm$ are pairwise disjoint subsets
of $A$ and $x_n \in E$ have norms not exceeding $1$ (compare \eqref{eqn:semi}). The set function
$\Mm \ni A \mapsto \|\mu\|_A \in [0,\infty]$ is called the \textit{semivariation} of $\mu$ (see \S4
of Chapter~I in \cite{din}) and known to have the following (simple) properties:
\begin{enumerate}[(SM1)]
\item $\|\mu\|_A \leqsl \|\mu\|_B$ provided $A, B \in \Mm$ are such that $A \subset B$;
\item $\|\mu\|_{\bigcup_{n=1}^{\infty} A_n} \leqsl \sum_{n=1}^{\infty} \|\mu\|_{A_n}$ for any
 collection of sets $A_n \in \Mm$.
\end{enumerate}
We divide the proof into a few separate cases.\par
First assume that
\begin{itemize}
\item[(C1)] every set $B \in \Mm$ with $\|\mu\|_B = \infty$ may be written in the form $B = B_1
 \cup B_2$ where $B_1, B_2 \in \Mm$ are pairwise disjoint and $\|\mu\|_{B_1} = \|\mu\|_{B_2} =
 \infty$.
\end{itemize}
Using (C1) and the induction argument, we easily find an infinite sequence of pairwise disjoint
sets $B_n \in \Mm$ for which $\|\mu\|_{B_n} = \infty$. So, it follows from the definition
of the semivariation that for each $n$ we may find finite systems $z_1^{(n)},\ldots,z_{N_n}^{(n)}
\in E$ of vectors whose norms are not greater than $1$ and disjoint sets $C_1^{(n)},\ldots,
C_{N_n}^{(n)} \in \Mm$ contained in $B_n$ such that
\begin{equation}\label{eqn:c1}
\Bigl\|\sum_{k=1}^{N_n} \mu(C_k^{(n)}) z_k^{(n)}\Bigr\| > 1.
\end{equation}
Now it suffices to arrange all sets $C_j^{(n)}$ in a sequence $A_1,A_2,\ldots$ and the vectors
$z_j^{(n)}$ in a corresponding sequence $x_1,x_2,\ldots$. Since the sets $A_n$ are pairwise
disjoint, we conclude from the definition of an i-measure that the series
\begin{equation}\label{eqn:series}
\sum_{n=1}^{\infty} \mu(A_n) x_n
\end{equation}
is unconditionally convergent (in the norm topology), which is contradictory to \eqref{eqn:c1}.
Thus, in that case the proof is complete.\par
Now we assume that there is a set $W \in \Mm$ with $\|\mu\|_W = \infty$ such that whenever $W = A
\cup B$ and $A, B \in \Mm$ are pairwise disjoint, then $\|\mu\|_A < \infty$ or $\|\mu\|_B < \infty$.
We then conclude from (SM1) that
\begin{itemize}
\item[(C2)] if $A, B \in \Mm$ are two disjoint subsets of $W$ and $\|\mu\|_A = \infty$, then
 $\|\mu\|_B < \infty$.
\end{itemize}
This case is divided into two subcases. First we additionally assume that there are a subset $V \in
\Mm$ of $W$ with $\|\mu\|_V = \infty$ and a number $\epsi > 0$ such that
\begin{itemize}
\item[(C3)] if $D \in \Mm$ is a subset of $V$ with $\|\mu\|_D = \infty$, then there is a set $B \in
 \Mm$ contained in $D$ for which $\epsi < \|\mu\|_B < \infty$.
\end{itemize}
Using (C3) for $D = V$, we may find a set $B_1 \in \Mm$ contained in $V$ such that $\epsi <
\|\mu\|_{B_1} < \infty$. We infer from (SM2) that $\|\mu\|_V \leqsl \|\mu\|_{V \setminus B_1}
+ \|\mu\|_{B_1}$ and hence $\|\mu\|_{V_1} = \infty$ for $V_1 \df V \setminus B_1$. Repeating this
reasoning for $D = V_1$, we may find a set $B_2 \in \Mm$ contained in $V_1$ for which $\epsi <
\|\mu\|_{B_2} < \infty$. Then $\|\mu\|_{V_2} = \infty$ for $V_2 \df V_1 \setminus B_2$. Continuing
this procedure, we obtain a sequence of pairwise disjoint sets $B_n \in \Mm$ such that
$\|\mu\|_{B_n} > \epsi$. Now repeating the reasoning from the previous case, we see that for
each $n$ there are finite systems $z_1^{(n)},\ldots,z_{N_n}^{(n)} \in E$ of vectors whose norms are
not greater than $1$ and $C_1^{(n)},\ldots,C_{N_n}^{(n)} \in \Mm$ of pairwise disjoint subsets
of $B_n$ such that $\|\sum_{k=1}^{N_n} \mu(C_k^{(n)}) z_k^{(n)}\| > \epsi$. As shown before,
this leads us to a contradiction with the fact that some series of the form \eqref{eqn:series} is
unconditionally convergent.\par
Finally, we add to (C2) the negation of (C3):
\begin{itemize}
\item[(C4)] whenever $V \in \Mm$ is a subset of $W$ with $\|\mu\|_V = \infty$ and $\epsi$ is
 a positive real number, then there exists a set $D = D(V,\epsi) \in \Mm$ contained in $V$ such that
 $\|\mu\|_D = \infty$ and every subset $B \in \Mm$ of $D$ with $\|\mu\|_B < \infty$ satisfies
 $\|\mu\|_B \leqsl \epsi$.
\end{itemize}
We now define by a recursive formula sets $V_n \in \Mm$: $V_0 \df D(W,1)$ and $V_n \df
D(V_{n-1},2^{-n})$ for $n > 0$. Put $V \df \bigcap_{n=0}^{\infty} V_n$ and, for $n > 0$, $L_n \df
V_{n-1} \setminus V_n$. Since the sets $V_n$ decrease, we see that
\begin{equation}\label{eqn:V0}
V_0 = V \cup \bigcup_{n=1}^{\infty} L_n.
\end{equation}
Further, it follows from (C2) that $\|\mu\|_{L_n} < \infty$ (because $\|\mu\|_{V_n} = \infty$ and
$L_n \cap V_n = \varempty$) and hence, thanks to the definition of $V_{n-1}$ (see (C4)),
$\|\mu\|_{L_n} \leqsl 2^{-n}$. So, (SM2) applied to \eqref{eqn:V0} gives
\begin{equation}\label{eqn:V}
\|\mu\|_V = \infty.
\end{equation}
Moreover, since $V \subset V_n$ for each $n$, we deduce from the property of $V_n$ specified in (C4)
that
\begin{itemize}
\item[($*$)] for every subset $B \in \Mm$ of $V$, $\|\mu\|_B \in \{0,\infty\}$.
\end{itemize}
We now fix a finite collection $A_1,\ldots,A_N \in \Mm$ of pairwise disjoint subsets of $V$ and
a corresponding system $x_1,\ldots,x_N \in E$ of vectors whose norms do not exceed $1$. We infer
from (C2) and ($*$) that there is an index $k \in \{1,\ldots,N\}$ such that $\|\mu\|_{A_j} = 0$ for
any $j \neq k$. Noticing that $\|\mu(A_j) x_j\| \leqsl \|\mu\|_{A_j}$, we get $\sum_{j=1}^N \mu(A_j)
x_j = \mu(A_k) x_k$. Consequently,
\begin{align*}
\|\mu\|_V &= \sup\{\|\mu(A) x\|\dd\ A \in \Mm,\ A \subset V,\ x \in E,\ \|x\| \leqsl 1\}\\
&= \sup\{\|\mu(A)\|\dd\ A \in \Mm,\ A \subset V\}.
\end{align*}
Since $\mu$ is a vector measure (by \LEM{op}), its range is a bounded set in $\llL(E,F)$ (see, for
example, Corollary~19 on page~9 in \cite{din}; a stronger property of countably additive vector
measures is the content of the Bartle-Dunford-Schwartz theorem, see Corollary~7 on page~14
in \cite{din}), and therefore the above formula contradicts \eqref{eqn:V}, which finishes the whole
proof.
\end{proof}

As a consequence of \THM{fin}, we obtain a generalisation of the Bartle-Dunford-Schwartz theorem
on the absolute continuity of vector measures with respect to some finite nonnegative measures (see,
for example, Corollary~6 on page~14 in \cite{d-u}). Below we continue the notation introduced
in the above proof.

\begin{cor}{abs}
If $\mu\dd \Mm \to \llL(E,F)$ is an i-measure, then there exists a measure $\lambda\dd \Mm \to
[0,\infty)$ such that the following condition is satisfied.
\begin{itemize}
\item[(ac)] For every $\epsi > 0$ there is $\delta(\epsi) > 0$ such that $\|\mu\|_A \leqsl \epsi$
 whenever $A \in \Mm$ satisfies $\lambda(A) \leqsl \delta(\epsi)$.
\end{itemize}
What is more, the measure $\lambda$ may be taken so that for each $A \in
\Mm$,
\begin{equation}\label{eqn:mut-abs}
0 \leqsl \lambda(A) \leqsl \|\mu\|_A.
\end{equation}
\end{cor}

Before giving a proof, we wish to emphasize that the above result is \textit{not} a special case
of the Bartle-Dunford-Schwartz theorem mentioned above, because the semivariation of an i-measure
is, in general, greater than the semivariation of a valued measure, defined in Definition~4
on page~2 in \cite{d-u}.

\begin{proof}
Let $\Gamma$ be the set of all finite systems $\gamma = (A_1,\ldots,A_N;x_1,\ldots,x_N)$ consisting
of pairwise disjoint sets $A_n \in \Mm$ and vectors $x_n \in E$ whose norms are not greater than
$1$. For each such $\gamma$ we define a set function $\mu_{\gamma}\dd \Mm \to F$ by $\mu_{\gamma}(B)
\df \sum_{j=1}^N \mu(B \cap A_j) x_j$ (provided $\gamma = (A_1,\ldots,A_N;x_1,\ldots,x_N)$). It is
easy to see that $\mu_{\gamma}$ is a vector measure. Observe also that
\begin{equation}\label{eqn:svar}
\sup_{\gamma\in\Gamma} \|\mu_{\gamma}(B)\| = \|\mu\|_B \qquad (B \in \Mm).
\end{equation}
The above formula, combined with \THM{fin}, yields that the collection
$\{\mu_{\gamma}\}_{\gamma\in\Gamma}$ is uniformly bounded. Further, let $A_n \in \Mm$ be pairwise
disjoint sets. We claim that
\begin{equation}\label{eqn:usa}
\lim_{n\to\infty} \|\mu\|_{A_n} = 0.
\end{equation}
Because if not, we may and do assume (after passing to a subsequence, if necessary) that
$\|\mu\|_{A_n} > \epsi$ for some positive real number $\epsi$ and all $n$. But this is impossible,
as shown in the proof of \THM{fin} (in the part concerning (C3)). So, \eqref{eqn:usa} holds which,
combined with \eqref{eqn:svar}, means that the collection $\{\mu_{\gamma}\}_{\gamma\in\Gamma}$ is
\textit{uniformly strongly additive} (consult Proposition~17 on page~8 in \cite{d-u}). We now deduce
from Corollary~5 on page~13 in \cite{d-u} that there is a measure $\lambda\dd \Mm \to [0,\infty)$
such that \eqref{eqn:mut-abs} holds and for any $\epsi > 0$ there is $\delta > 0$ for which
$\sup_{\gamma\in\Gamma} \|\mu_{\gamma}(A)\| \leqsl \epsi$ provided $\lambda(A) \leqsl \delta$.
So, a look at \eqref{eqn:svar} finishes the proof.
\end{proof}

Whenever $\mu$ is an i-measure and $\lambda$ is a probabilistic measure, both defined on a common
$\sigma$-algebra, we shall write $\mu \ll \lambda$ if (ac) is fulfilled.\par
As a consequence of \COR{abs}, we obtain the following generalisation of a theorem of Pettis (see
Theorem~1 on page~10 in \cite{d-u}).

\begin{cor}{abs-abs}
For an i-measure $\mu\dd \Mm \to \llL(E,F)$ and a measure $\nu\dd \Mm \to [0,\infty)$, $\mu \ll \nu$
iff $\mu$ vanishes on all sets on which $\nu$ vanishes.
\end{cor}
\begin{proof}
The `only if' part is immediate. To show the `if' part, assume $\mu$ vanishes on all sets on which
$\nu$ vanishes. By \COR{abs}, there exists a measure $\lambda\dd \Mm \to [0,\infty)$ such that
\begin{equation}\label{eqn:ll}
\mu \ll \lambda
\end{equation}
and \eqref{eqn:mut-abs} is fulfilled. We infer from the latter condition that $\lambda(A) = 0$ iff
$\|\mu\|_A = 0$. But $\|\mu\|_A = 0$ \iaoi{} $\mu$ vanishes on all measurable subsets of $A$.
We conclude that if $\nu(A) = 0$, then $\lambda(A) = 0$. So, it follows from the Radon-Nikodym
theorem that there exists an $\Mm$-measurable function $g\dd Z \to [0,\infty)$ (where $Z$ is the set
on which $\Mm$ is a $\sigma$-algebra) such that $\lambda(A) = \int_A g \dint{\nu}$ for all $A \in
\Mm$. In particular, $g$ is $\nu$-integrable and therefore for any $\epsi > 0$ there exists $\delta
> 0$ such that $\int_A g \dint{\nu} \leqsl \epsi$ provided $\nu(A) \leqsl \delta$. This property,
combined with \eqref{eqn:ll}, finishes the proof.
\end{proof}

\begin{rem}{fin}
From \THM{fin} one may deduce the following result, which, due to the knowledge of the author, is
new:
\begin{quote}
\textit{The variation of a vector measure $\mu\dd \Mm \to E$ is a finite measure iff
$\sum_{n=1}^{\infty} \|\mu(A_n)\| < \infty$ for any countable collection of pairwise disjoint sets
$A_n \in \Mm$.}
\end{quote}
The necessity is immediate, while the sufficiency follows from the fact that a measure satisfying
the condition formulated above may naturally be identified with an i-measure, as described
below.\par
Assume $\mu\dd \Mm \to E$ is a vector measure which satisfies the above condition. Since every
vector $x$ of $E$ naturally induces a (continuous) linear operator from $\KKK$ into $E$ (which sends
$1$ to $x$), we may identify $\mu$ with a set function of $\Mm$ into $\llL(\KKK,E)$. Under such
an identification, $\mu$ turns out to be an i-measure whose total semivariation is equal
to the total variation of $\mu$, regarded as an $E$-valued set function. We leave the details
to the reader.
\end{rem}

The book \cite{din} is devoted to integration of vector-valued functions with respect
to vector-valued set functions. Below we adapt this concept to define integration with respect
to i-measures, which turns out to be much easier and more elegant.

\begin{dfn}{vint}
Let $\mu\dd \Mm \to \llL(E,F)$ be an i-measure defined on a $\sigma$-algebra $\Mm$ of subsets
of a set $Z$. Denote by $S_{\Mm}(Z,E)$ the set of all functions $f \in \ell_{\infty}(Z,E)$ such that
the set $f(Z)$ is countable and $f^{-1}(\{e\}) \in \Mm$ for any $e \in E$. It is easy to see that
$S_{\Mm}(Z,E)$ is a linear subspace of $\ell_{\infty}(Z,E)$. For any $f \in S_{\Mm}(Z,E)$ we define
\begin{equation}\label{eqn:i-int}
\int_Z f \dint{\mu} = \int_Z f(z) \dint{\mu(z)} \df \sum_{e \in E} \mu(f^{-1}(\{e\})) e
\end{equation}
(the above series is unconditionally convergent; see \DEF{i-m}) and call $\int_Z f \dint{\mu}$
the \textit{integral} of $f$ with respect to $\mu$.\par
The uniform closure of $S_{\Mm}(Z,E)$ coincides with $M_{\Mm}(Z,E)$ (see \DEF{measurable}).
\end{dfn}

Our aim is to extend the integral defined above from $S_{\Mm}(X,E)$ to $M_{\Mm}(X,E)$. This is
enabled thanks to \THM{fin} and the following

\begin{lem}{cont}
For every i-measure $\mu\dd \Mm \to \llL(E,F)$ \textup{(}where $\Mm$ is a $\sigma$-algebra on a set
$Z$\textup{)}, the operator $T\dd S_{\Mm}(Z,E) \ni f \mapsto \int_Z f \dint{\mu} \in F$ is linear
and continuous. Moreover, $\|T\| = \|\mu\|_Z$.
\end{lem}

A simple proof of \LEM{cont} is left to the reader.

\begin{dfn}{VINT}
Let $\mu\dd \Mm \to \llL(E,F)$ be an i-measure defined on a $\sigma$-algebra $\Mm$ of subsets
of a set $Z$. For any $f \in M_{\Mm}(Z,E)$, the \textit{integral} $\int_Z f\dint{\mu} = \int_Z f(z)
\dint{\mu(z)}$ of $f$ with respect to $\mu$ is defined as $\bar{T} f$ where $\bar{T}\dd M_{\Mm}(Z,E)
\to F$ is the unique continuous extension of $T\dd S_{\Mm}(Z,E) \ni f \mapsto \int_Z f \dint{\mu}
\in F$. Then $\|\int_Z f \dint{\mu}\| \leqsl \|f\| \cdot \|\mu\|_Z$ for any $f \in M_{\Mm}(Z,E)$.
\end{dfn}

Our main result on i-measures is the following

\begin{thm}[Bounded Weak Convergence Theorem]{bwc}
Let $\mu\dd \Mm \to \llL(E,F)$ be an i-measure \textup{(}where $\Mm$ is a $\sigma$-algebra on a set
$Z$\textup{)}. If $f_n \in M_{\Mm}(Z,E)$ are uniformly bounded and converge pointwise to $f\dd Z \to
E$ in the weak topology of $E$, then $\int_Z f_n \dint{\mu}$ converge to $\int_Z f \dint{\mu}$
in the weak topology of $F$.
\end{thm}

The main difficulty in the proof of the above result is that sequences which weakly converge to $0$
may consist of unit vectors. We precede the proof of \THM{bwc} by a few auxiliary results. From now
until the end of the proof, $Z$, $\Mm$ and $\mu$ are as specified in \THM{bwc}.\par
We begin with a counterpart of the Bartle Bounded Convergence \cite{bar} (see also Theorem~1
on page~56 in \cite{d-u}) for i-measures. Other results in the topic of bounded and dominated
convergence theorems the reader may find in \cite{t-w} and \cite{s-t}.

\begin{thm}[Bounded Norm Convergence Theorem]{bnc}
If $f_n \in M_{\Mm}(Z,E)$ are uniformly bounded and converge pointwise to $f\dd Z \to E$ in the norm
topology of $E$, then $\int_Z f_n \dint{\mu}$ converge to $\int_Z f \dint{\mu}$ in the norm topology
of $F$.
\end{thm}
\begin{proof}
We mimic the proof of the Bartle Convergence Theorem presented in \cite{d-u}. It follows from
\COR{abs} that there is a probabilistic measure $\lambda\dd \Mm \to [0,1]$ such that (ac) holds.
We need to show that $\|\int_Z g_n \dint{\mu}\|$ converge to $0$ for $g_n \df f_n - f$. For each
$n$, there is $u_n \in S_{\Mm}(X,E)$ such that $\|g_n - u_n\| < 2^{-n}$ and $\|\int_Z g_n \dint{\mu}
- \int_Z u_n \dint{\mu(z)}\| < 2^{-n}$. We conclude that it suffices to show that $\|\int_Z u_n
\dint{\mu}\|$ converge to $0$. Note that the functions $u_n$ are uniformly bounded and converge
pointwise to $0$ in the norm topology of $E$. Suppose $\|u_n(z)\| \leqsl C$ for all $n$ and $z \in
Z$ (and a positive constant $C$). Fix $\epsi > 0$ and put $\delta = \delta(\epsi/C)$ (see (ac)).
It follows from Egoroff's theorem that there exists a set $A \in \Mm$ such that $\lambda(A) \leqsl
\delta$ and the functions $u_n$ converge uniformly to $0$ on $Z \setminus A$. So, denoting
(as usual) by $j_A$ and $j_{Z \setminus A}$ the characteristic functions of $A$ and $Z \setminus A$
(respectively), we see that the functions $j_{Z \setminus A} u_n$ converge uniformly to $0$.
Consequently, $\lim_{n\to\infty} \|\int_Z j_{Z \setminus A} u_n \dint{\mu}\| = 0$. Further,
it follows from the definition of the vector integral that $\|\int_Z j_A u_n \dint{\mu}\| \leqsl
\|u_n\| \cdot \|\mu\|_A$. Finally, from the choice of $A$ and $\delta$ we infer that $\|\mu\|_A
\leqsl \epsi/C$ and therefore
\begin{equation*}
\limsup_{n\to\infty} \Bigl\|\int_Z u_n \dint{\mu}\Bigr\| \leqsl \limsup_{n\to\infty} \Bigl\|\int_Z
j_A u_n \dint{\mu}\Bigr\| + \limsup_{n\to\infty} \Bigl\|\int_Z j_{Z \setminus A} u_n
\dint{\mu}\Bigr\| \leqsl \epsi
\end{equation*}
and we are done.
\end{proof}

\begin{lem}{A}
Let $Y$ be a compact metrisable space and $u_n$ be members of $M(Y,E)$. Then the set $S$ of all
$y \in Y$ for which $u_n(y)$ converge to $0$ in the weak topology of $E$ is coanalytic
\textup{(}in the sense of Suslin\textup{)}.
\end{lem}
\begin{proof}
Let us recall that $S$ is coanalytic provided $Y \setminus S$ coincides with the image of a Borel
subset of $Y \times [0,1]$ under a continuous function, which we shall now show.\par
Let $E_0$ be the closed linear span of the set $\bigcup_{n=1}^{\infty} u_n(Y)$. Since $E_0$ is
separable, there exists an isometric linear operator $U\dd E_0 \to C([0,1],\KKK)$. Since sequences
of elements of $C([0,1],\KKK)$ which converge to $0$ in the weak topology (of $C([0,1],\KKK)$) are
simply characterised, we infer that for an arbitrary sequence of elements $z_n$ of $E$,
\begin{itemize}
\item[(w)] $z_n$ converge to $0$ in the weak topology of $E$ iff $U z_n$ converge pointwise to $0$.
\end{itemize}
Now define $v_n\dd Y \times [0,1] \to E$ by $v_n(y,t) \df U(u_n(y))(t)$. Then $v_n \in M(Y \times
[0,1],\KKK)$ (compare \LEM{2}), which means, by the metrisability of $Y$, that $v_n$ are Borel.
We conclude that the set $B \df \{(y,t) \in Y \times [0,1]\dd\ \lim_{n\to\infty} v_n(y,t) = 0\}$ is
Borel in $Y \times [0,1]$. Note that (w) implies that
\begin{equation*}
y \in S \iff \{y\} \times [0,1] \subset B.
\end{equation*}
So, denoting by $\pi\dd Y \times [0,1] \to Y$ the natural projection, we see that $S = Y \setminus
\pi((Y \times [0,1]) \setminus B)$, which finishes the proof.
\end{proof}

\begin{lem}{B}
Let $u_n \in S_{\Mm}(Z,E)$ be uniformly bounded and converge pointwise to $0$ in the weak topology
of $E$ and let $\psi \in F^*$. There exist a compact metrisable space $Y$, an i-measure $\nu\dd
\Mm(Y) \to \llL(E,F)$, and uniformly bounded functions $v_n \in M(Y,E)$ which converge pointwise
to $0$ in the weak topology of $E$ and satisfy \textup{(}for each $n$\textup{)}
\begin{equation}\label{eqn:same}
\psi\Bigl(\int_Z u_n \dint{\mu}\Bigr) = \psi\Bigl(\int_Y v_n \dint{\nu}\Bigr).
\end{equation}
\end{lem}
\begin{proof}
Denote by $\bbB$ the collection of all nonempty sets of the form $u_n^{-1}(\{e\})$ where $n$ and
$e \in E$ are arbitrary. Observe that $\bbB$ is countable (and nonempty). So, we may arrange all
members of $\bbB$ in an infinite sequence $A_1,A_2,\ldots$ (repeating, if necessary, some of them).
For simplicity, let $j_n\dd Z \to \{0,1\}$ stand for the characteristic function of $A_n$. Put $Y
\df \{0,1\}^{\omega}$ (that is, $Y$ is the infinite countable power of $\{0,1\}$) and equip $Y$ with
the product topology. Define $\Phi\dd Z \to Y$ by $\Phi(z) \df (j_n(z))_{n=1}^{\infty}$. It is easy
to see that $\Phi^{-1}(B) \in \Mm$ for any $B \in \Mm(Y)$ (since $Y$ is metrisable, $\Mm(Y)$
consists of all Borel sets in $Y$). Further, let $\nu\dd \Mm(Y) \to \llL(E,F)$ be given by $\nu(B)
\df \mu(\Phi^{-1}(B))$. It is readily seen that $\nu$ is an i-measure such that $\|\nu\|_Z \leqsl
\|\mu\|_Z$. Further, we put $Z' \df \Phi(Z)$ and $Y_m \df \{(y_n)_{n=1}^{\infty} \in Y\dd\ y_m = 1\}
(\in \Mm(Y))$. Observe that
\begin{equation}\label{eqn:trace}
\Phi(A_n) = Z' \cap Y_n
\end{equation}
for any $n > 0$. We claim that there exist uniformly bounded functions $w_n \in M(Y,E)$ such that
for any superset $C \in \Mm(Y)$ of $Z'$ and each $n$,
\begin{equation}\label{eqn:u-w}
\int_Z u_n \dint{\mu} = \int_Y j_C(y) w_n(y) \dint{\nu(y)}
\end{equation}
where $j_C\dd Y \to \{0,1\}$ is the characteristic function of $C$. We may define the functions
$w_n$ as follows. Fix $n$ and for simplicity put (for a moment) $u = u_n$. Write $u(Y) = \{e_1,e_2,
\ldots\}$ where the vectors $e_k$ are distinct (so, there can be finitely many such vectors) and
denote by $m_k$ a natural number such that $A_{m_k} = u^{-1}(\{e_k\})$. Notice that the sets
$A_{m_1},A_{m_2},\ldots$ are pairwise disjoint and cover $Z$. It follows from the definition
of $\Phi$ that hence also the sets $\Phi(A_{m_1}),\Phi(A_{m_2}),\ldots$ are pairwise disjoint
(although $\Phi$ may not be one-to-one). Thus, we infer from \eqref{eqn:trace} that there are
\textit{pairwise disjoint} sets $B_k \in \Mm(Y)$ such that
\begin{equation}\label{eqn:Bk}
B_k \cap Z' = \Phi(A_{m_k}).
\end{equation}
We define $w_n$ by the rules: $w_n(y) = e_k$ for $y \in B_k$ and $w_n(y) = 0$ if $y \notin \bigcup_k
B_k$. Since $w_n(Y) \subset u_n(Y) \cup \{0\}$, we see that the functions $w_n$ are uniformly
bounded (it is also clear that they belong to $M(Y,E)$). Let us briefly check \eqref{eqn:u-w}.
If $Z ' \subset C \in \Mm(Y)$ and $w \df j_C w_n$, then (under the above notation) $\Phi(A_{m_k}) =
(B_k \cap C) \cap Z'$, thanks to \eqref{eqn:Bk}. So, $\Phi^{-1}(w^{-1}(\{e_k\})) = A_{m_k}$ provided
$e_k \neq 0$. Hence
\begin{equation*}
\int_Y w \dint{\nu} = \sum_{e \in E} \nu(w^{-1}(\{e\})) e = \sum_{e_k \neq 0} \mu(A_{m_k}) e_k
= \int_Z u_n \dint{\mu},
\end{equation*}
which finishes the proof of \eqref{eqn:u-w}.\par
Now let $S$ consist of all $y \in Y$ for which $w_n(y)$ converge to $0$ in the weak topology of $E$.
It follows from \LEM{A} that $S$ is coanalytic. Observe that $w_n \circ \Phi = u_n$ (thanks
to \eqref{eqn:Bk}) and therefore $Z' \subset S$. Denote by $\nu_{\psi}\dd \Mm(Y) \to \llL(E,\KKK) =
E^*$ the i-measure given by $\nu_{\psi}(A) = \psi \circ \nu(A)$. Now let $\lambda\dd \Mm \to
[0,\infty]$ be the so-called \textit{variation} of $\nu_{\psi}$; that is,
\begin{equation*}
\lambda(A) = \sup\Bigl\{\sum_{n=1}^{\infty} \|\nu_{\psi}(A_n)\|\dd\ A_n \in \Mm \textup{ are
pairwise disjoint subsets of } A\Bigr\}.
\end{equation*}
It follows from Proposition~4 (on page~54) in \S4 of Chapter~I in \cite{din} (and may easily be
checked) that $\lambda(Z) \leqsl \|\nu\|_Z$. So, $\lambda$ is a finite measure. Since coanalytic
sets are \textit{measurable} with respect to any finite Borel measure (consult, for example,
Theorem~A.13 in \cite{tak}; see also Theorem~1 in \S4 of Chapter~XIII in \cite{k-m}), we deduce that
there are two sets $A, B \in \Mm(Y)$ such that $A \subset S \subset B$ and $\lambda(B \setminus A) =
0$. Consequently, $B \supset Z'$ and thus \eqref{eqn:u-w} holds for $C = B$. We put $v_n \df j_A w_n
\in M(Y,E)$. We see that the functions $v_n$ are uniformly bounded and converge pointwise to $0$
in the weak topology of $E$, since $A \subset S$. To show \eqref{eqn:same}, we note that $\int_Y
(v_n - j_B w_n) \dint{\nu_{\psi}} = 0$ (since $\lambda(B \setminus A) = 0$) and hence
\begin{multline*}
\psi\Bigl(\int_Y v_n \dint{\nu}\Bigr) = \psi\Bigl(\sum_{e \in E} \nu(v_n^{-1}(\{e\})) e\Bigr) =
\sum_{e \in E} \nu_{\psi}(v_n^{-1}(\{e\})) e = \int_Y v_n \dint{\nu_{\psi}}\\
= \int_Y j_B w_n \dint{\nu_{\psi}} = \psi\Bigl(\int_Y j_B w_n \dint{\nu}\Bigr) = \psi\Bigl(\int_Z
u_n \dint{\mu}\Bigr).
\end{multline*}
\end{proof}

\begin{lem}{C}
Let $Y$ be a compact space and $\nu\dd \Mm(Y) \to \llL(E,F)$ be an i-measure. Let $T\dd C(Y,E) \to
F$ be given by $T f \df \int_Y f \dint{\nu}$ and let $\bar{T}\dd M(Y,E) \to F^{**}$ be as specified
in \COR{C-M}. Then $\|T\| = \|\nu\|_Y$ and
\begin{equation}\label{eqn:Tint}
\bar{T} f = \int_Y f \dint{\nu} \qquad (f \in M(Y,E)).
\end{equation}
In particular, $\bar{T}\dd M(Y,E) \to F$.
\end{lem}
\begin{proof}
Denote by $S f$ the right-hand side expression of \eqref{eqn:Tint}. Then $S\dd M(Y,E) \to F$ is
linear, continuous and $\|S\| = \|\nu\|_Y$. So, to conclude the whole assertion, it suffices to show
that $S = \bar{T}$. Since the weak topology of $F$ coincides with the topology on $F$ inherited from
the weak* topology of $F^{**}$, we infer from \THM{bnc} and \COR{C-M} that for $L \df S$ as well as
$L \df \bar{T}$ one has
\begin{itemize}
\item[(bc*)] whenever $u_n \in M(Y,E)$ are uniformly bounded and converge pointwise to $u\dd Y \to
 E$ in the norm topology of $E$, then $L u_n$ converge to $L u$ in the weak* topology of $F^{**}$.
\end{itemize}
Further, the proof of \LEM{1} shows that $M(Y,E)$ coincides with the smallest set among all $B
\subset \ell_{\infty}(Y,E)$ which include $C(Y,E)$ and satisfy (M1') with $Y$ inserted in place
of $X$ (see the proof of \LEM{1}). So, we easily infer from this property and from (bc*) that $S =
\bar{T}$.
\end{proof}

\begin{proof}[Proof of \THM{bwc}]
We begin similarly as in the proof of \THM{bnc}: it is enough to show that $\int_Z g_n \dint{\mu}$
converge to $0$ in the weak topology of $F$ for $g_n \df f_n - f$. For each $n$, there is $u_n \in
S_{\Mm}(X,E)$ such that $\|g_n - u_n\| < 2^{-n}$ and $\|\int_Z g_n \dint{\mu} - \int_Z u_n
\dint{\mu(z)}\| < 2^{-n}$. We conclude that it suffices to show that $\int_Z u_n \dint{\mu}$
converge to $0$ in the weak topology of $F$. Note that the functions $u_n$ are uniformly bounded and
converge pointwise to $0$ in the weak topology of $E$. Let $\psi \in F^*$. We only need to show that
\begin{equation}\label{eqn:lim}
\lim_{n\to\infty}\psi\Bigl(\int_Z u_n \dint{\mu}\Bigr) = 0.
\end{equation}
It follows from \LEM{B} that we may and do assume $Z$ is a compact topological space and $\Mm =
\Mm(Z)$. Let $T\dd C(Z,E) \to F$ be given by $T f \df \int_Z \dint{\mu}$ and let $\bar{T}\dd M(Z,E)
\to F^{**}$ be as specified in \COR{C-M}. We infer from \LEM{C} that $\int_Z u_n \dint{\mu} =
\bar{T} u_n$, and from (BC*) that $\bar{T} u_n$ converge pointwise to $0$ in the weak* topology
of $F^{**}$. Consequently, \eqref{eqn:lim} is fulfilled.
\end{proof}

\begin{rem}{integrable}
\THM{bnc} enables us to define vector-valued integrable functions with respect to i-measures.
Namely, if $\mu\dd \Mm \to \llL(E,F)$ is an i-measure on a set $Z$ and $g\dd Z \to E$ is
an arbitrary $\Mm$-measurable (in the sense of \DEF{measurable}) function, we put $\Dd(g) \df \{A
\in \Mm\dd\ j_A g \in \ell_{\infty}(Z,E)\}$. Notice that $\Dd(g)$ is an ideal in $\Mm$ such that
every set $A$ in $\Mm$ is a countable union of members of $\Dd(g)$. We call the function $g$
\textit{integrable} if the set function $\nu\dd \Dd(g) \ni A \mapsto \int_Z j_A g \dint{\mu} \in F$
extends to a (necessarily unique) vector measure $\bar{\nu}\dd \Mm \to F$. If this happens, for each
$A \in \Mm$ we define the \textit{integral} $\int_A g \dint{\mu}$ (of $g$ on $A$ with respect
to $\mu$) as $\bar{\nu}(A)$. Notice that in the above situation, the set function $\nu$ is always
a \textit{conditional} vector measure; that is, if $A_n \in \Dd(g)$ are pairwise disjoint and
$\bigcup_{n=1}^{\infty} A_n \in \Dd(g)$, then $\nu(\bigcup_{n=1}^{\infty} A_n) = \sum_{n=1}^{\infty}
\nu(A_n)$, which follows from \THM{bnc}. In particular, every bounded $\Mm$-measurable function is
integrable. One may show that integrable functions form a vector space and the integral $\int_A$
(with respect to $\mu$) is a linear operator (for each $A \in \Mm$). We will not develop this
concept here---this remark has only an introductory character.
\end{rem}

\section{Weak* i-measures}

This part is devoted to generalisation of the concept of i-measures to the context of weak*
topologies of dual Banach spaces and to give representations of continuous linear operators from
$C(X,E)$ into dual Banach spaces. To make the presentation simple and transparent, for $T \in
\llL(E,F^*)$ and $f \in F$ we shall write $\scalar{f}{T(\cdot)}$ to denote the functional $E \ni x
\mapsto (T x) f \in \KKK$.\par
We begin with

\begin{dfn}{w*}
A series $\sum_{n=1}^{\infty} T_n$ with summands in $\llL(E,F^*)$ is said to be
\textit{independently w*-convergent} if the series $\sum_{n=1}^{\infty} \scalar{f}{T_n(\cdot)}$
(of elements of $\llL(E,\KKK)$) is independently convergent for every $f \in F$. A \textit{weak*
i-measure} is a set function $\mu\dd \Mm \to \llL(E,F^*)$ (where $\Mm$ is a $\sigma$-algebra
on a set $Z$) if for any $f \in F$, the set function $\Mm \ni A \mapsto \scalar{f}{\mu(A)(\cdot)}
\in \llL(E,\KKK)$ is an i-measure. Equivalently, $\mu$ is a weak* i-measure iff
\begin{equation*}
\scalar{f}{\mu\Bigl(\bigcup_{n=1}^{\infty} A_n\Bigr)(\cdot)} = \sum_{n=1}^{\infty}
\scalar{f}{\mu(A_n)(\cdot)}
\end{equation*}
(for any $f \in F$) and the series $\sum_{n=1}^{\infty} \mu(A_n)$ is independently w*-convergent for
any collection of pairwise disjoint sets $A_n \in \Mm$. The \textit{total semivariation} $\|\mu\|_Z
\in [0,\infty]$ of a weak* i-measure is defined by the formula \eqref{eqn:semi}, as for i-measures.
\end{dfn}

As for i-measures, it turns out that

\begin{pro}{fin}
Every weak* i-measure has finite total semivariation.
\end{pro}

In the proof we shall need the following elementary result, whose proof is given for the sake
of completeness.

\begin{lem}{compl}
For any $\sigma$-algebra $\Mm$ on a set $Z$ and Banach spaces $E$ and $F$, the set
$\MmM(\Mm,\llL(E,F))$ is a Banach space when the algebraic operations are defined pointwise and
the norm is a function which assigns to each i-measure its total semivariation.
\end{lem}
\begin{proof}
It is readily seen that $\MmM(\Mm,\llL(E,F))$ is a vector space and the function $\|\cdot\|_Z$ is
a norm (thanks to \THM{fin}). Take a Cauchy sequence of i-measures $\mu_n\dd \Mm \to \llL(E,F)$. For
any $A \in \Mm$ we have $\|\mu_n(A) - \mu_m(A)\| \leqsl \|\mu_n - \mu_m\|_Z$ and therefore $\mu(A)
\df \lim_{n\to\infty} \mu_n(A)$ is a well defined member of $\llL(E,F)$. In this way we have
obtained a set function $\mu\dd \Mm \to \llL(E,F)$. It is immediate that $\mu$ is finitely additive.
For any $\epsi > 0$, choose $\nu_{\epsi}$ such that
\begin{equation*}
\|\mu_n - \mu_m\|_Z \leqsl \frac12 \epsi
\end{equation*}
for all $n, m \geqsl \nu_{\epsi}$. Fix a countable collection of pairwise disjoint sets $A_k \in
\Mm$ and a sequence of vectors $x_k \in E$ whose norms do not exceed $1$. For $n, m \geqsl
\nu_{\epsi}$ and arbitrary $N$ and $M$ we have $\|\sum_{k=N}^{N+M} \mu_n(A_k) x_k - \mu_m(A_k) x_k\|
\leqsl \|\mu_n - \mu_m\|_Z \leqsl \frac12 \epsi$. So, letting $m \to \infty$, we get
\begin{equation}\label{eqn:fund}
\Bigl\|\sum_{k=N}^{N+M} \mu_n(A_k) x_k - \sum_{k=N}^{N+M} \mu(A_k) x_k\Bigr\| \leqsl \frac12 \epsi
\qquad (n \geqsl \nu_{\epsi}).
\end{equation}
This, in particular, yields that $\|\mu_n - \mu\|_Z \leqsl \frac12 \epsi$ for $n \geqsl \nu_{\epsi}$
and consequently $\lim_{n\to\infty} \|\mu_n - \mu\|_Z = 0$, provided $\mu$ is an i-measure. Further,
for $n = \nu_{\epsi}$ the series $\sum_{k=1}^{\infty} \mu_n(A_k) x_k$ is convergent, hence there is
$N_0$ such that $\|\sum_{k=N}^{N+M} \mu_n(A_k) x_k\| \leqsl \frac12 \epsi$ whenever $N \geqsl N_0$
and $M > 0$. This inequality, combined with \eqref{eqn:fund}, gives $\|\sum_{k=N}^{N+M} \mu(A_k)
x_k\| \leqsl \epsi$ for any $N \geqsl N_0$ and $M > 0$. We conclude that the series
$\sum_{k=1}^{\infty} \mu(A_k) x_k$ is convergent. Finally, when $x_k = x \in E$ for each $k$ (where
$\|x\| \leqsl 1$), $A_1 = B \in \Mm$ and $A_k = \varempty$ for all $k > 1$, \eqref{eqn:fund} gives
\begin{equation}\label{eqn:single}
\|\mu_n(B) x - \mu(B) x\| \leqsl \frac12 \epsi \qquad (n \geqsl \nu_{\epsi}).
\end{equation}
So, if (again) the sets $A_k$ are pairwise disjoint and $A \df \bigcup_{k=1}^{\infty} A_k$, then
for $n = \nu_{\epsi}$ there is $M$ such that $\|\mu_n(A \setminus \bigcup_{k=1}^N A_k) x\| \leqsl
\frac12 \epsi$ for any $N \geqsl M$. Putting $B = A \setminus \bigcup_{k=1}^N A_k$ (and $n =
\nu_{\epsi}$) in \eqref{eqn:single}, we deduce that $\|\mu(A \setminus \bigcup_{k=1}^N A_k) x\|
\leqsl \epsi$ for any $N \geqsl M$. Thus, $\lim_{n\to\infty} \|\mu(A \setminus \bigcup_{k=1}^n A_k)
x\| = 0$, which means that $\mu$ is countably additive and consequently $\mu \in
\MmM(\Mm,\llL(E,F))$.
\end{proof}

\begin{proof}[Proof of \PRO{fin}]
Let $\mu\dd \Mm \to \llL(E,F^*)$ be a weak* i-measure defined on a $\sigma$-algebra $\Mm$ of subsets
of a set $Z$. For any $f \in F$ define $\mu_f\dd \Mm \to \llL(E,\KKK)$ by $\mu_f(A) \df
\scalar{f}{\mu(A)(\cdot)}$. We infer from the definition of a weak* i-measure that $\mu_f \in
\MmM(\Mm,\llL(E,F))$ and from \LEM{compl} that $\MmM(\Mm,\llL(E,F))$ is a Banach space. So,
we conclude from the Closed Graph Theorem that a linear operator $\Phi\dd F \ni f \mapsto \mu_f \in
\MmM(\Mm,\llL(E,F))$ is continuous (it is obvious that the graph of $\Phi$ is closed). Hence,
$M \df \sup\{\|\mu_f\|_Z\dd\ f \in F,\ \|f\| \leqsl 1\} < \infty$. Now take a collection of $N$
pairwise disjoint sets $A_n \in \Mm$ and a corresponding system of vectors $x_n \in E$ whose norms
do not exceed $1$. Then
\begin{align*}
\Bigl\|\sum_{n=1}^N \mu(A_n) x_n\Bigr\| &= \sup\Bigl\{\Bigl|\sum_{n=1}^N (\mu(A_n) x_n)(f)\Bigr|\dd\
f \in F,\ \|f\| \leqsl 1\Bigr\}\\
&= \sup\Bigl\{\Bigl|\sum_{n=1}^N \mu_f(A_n) x_n\Bigr|\dd\ f \in F,\ \|f\| \leqsl 1\Bigr\} \leqsl M
\end{align*}
and thus $\|\mu\|_Z \leqsl M$.
\end{proof}

\begin{exm}{nonabs}
One may hope (being inspired by \COR{abs} and \PRO{fin}) that for every weak* i-measure $\mu$ there
is a nonnegative real-valued measure $\lambda$ such that $\mu$ vanishes on all sets on which
$\lambda$ vanishes. As the following example shows, in some cases this is very far from
the truth.\par
Let $Z$ be an uncountable set and $\Mm$ the $\sigma$-algebra of all subsets of $Z$. Let $E = \KKK$
and $F = \ell_1(Z,\KKK)$. Then $F^* = \ell_{\infty}(Z,\KKK)$. Further, for any set $A \in \Mm$ let
$\mu(A)\dd E \to F$ be given by $\mu(A)\lambda = \lambda j_A$ where, as usual, $j_A$ is
the characteristic function of $A$. We see that $\mu\dd \Mm \to \llL(E,F^*)$. Observe that $\mu(A) =
0$ iff $A = \varempty$ and thus there is no measure $\lambda\dd \Mm \to [0,\infty)$ for which $\mu
\ll \lambda$ (because $Z$ is uncountable). However, $\mu$ is a weak* i-measure, which may simply be
verified.
\end{exm}

Let $\mu\dd \Mm \to \llL(E,F^*)$ be a weak* i-measure defined on a $\sigma$-algebra $\Mm$ of subsets
of a set $Z$. For $f \in S_{\Mm}(Z,E)$, we define the \textit{weak* integral} $\int^{w*}_Z f
\dint{\mu}$ of $f$ with respect to $\mu$ as the right-hand side expression of \eqref{eqn:i-int},
understood in the weak* topology of $F^*$; that is,
\begin{equation*}
\Bigl(\int^{w*}_Z f \dint{\mu}\Bigr)(v) = \sum_{e \in E} \Bigl(\mu(f^{-1}(\{e\}))e\Bigr)(v) \qquad
(v \in F).
\end{equation*}
We see (as for i-measures) that the operator $L\dd S_{\Mm}(Z,E) \ni f \mapsto \int^{w*}_Z f
\dint{\mu} \in F^*$ is linear and continuous, and $\|L\| = \|\mu\|_Z$ (because the norm of $F^*$ is
lower semicontinuous with respect to the weak* topology). We extend the operator $L$ to the whole
$M_{\Mm}(Z,E)$ and for $f \in M_{\Mm}(Z,E)$ use $\int^{w*}_Z f \dint{\mu}$ to denote the value
at $f$ of the unique continuous extension of $L$, which is called the \textit{weak* integral} of $f$
(with respect to $\mu$). We see that $\|\int^{w*}_Z f \dint{\mu}\| \leqsl \|f\| \cdot \|\mu\|_Z$.
Note also that if the weak* i-measure is actually an i-measure, then $\int^{w*}_Z f \dint{\mu} =
\int_Z f \dint{\mu}$ for any $f \in M_{\Mm}(Z,E)$. We also have:

\begin{thm}[Bounded Weak* Convergence Theorem]{bw*c}
Let $\mu\dd \Mm \to \llL(E,F^*)$ be a weak* i-measure \textup{(}where $\Mm$ is a $\sigma$-algebra
on a set $Z$\textup{)}. If $f_n \in M_{\Mm}(Z,E)$ are uniformly bounded and converge pointwise
to $f\dd Z \to E$ in the weak topology of $E$, then $\int^{w*}_Z f_n \dint{\mu}$ converge
to $\int^{w*}_Z f \dint{\mu}$ in the weak* topology of $F^*$.
\end{thm}
\begin{proof}
Fix $v \in F$. We need to show that $(\int^{w*}_Z f_n \dint{\mu})(v)$ converge to $(\int^{w*}_Z f
\dint{\mu})(v)$. Define $\nu\dd \Mm \to \llL(E,\KKK) = E^*$ by $\nu(A) \df
\scalar{v}{\mu(A)(\cdot)}$. It follows from the definition of a weak* i-measure that $\nu$ is
an i-measure. What is more, $\|\nu\|_Z \leqsl \|\mu\|_Z$ and
\begin{equation}\label{eqn:m-n}
\Bigl(\int^{w*}_Z u \dint{\mu}\Bigr)(v) = \int_Z v \dint{\nu} \quad (\in \KKK)
\end{equation}
for any $u \in M_{\Mm}(Z,E)$ (this is clear for $u \in S_{\Mm}(Z,E)$ and for arbitrary $u$ follows
from the continuity in $u$ of both sides of \eqref{eqn:m-n}). So, the assertion of the theorem
follows from \eqref{eqn:m-n} and \THM{bwc} applied for $\nu$.
\end{proof}

In some cases weak* i-measures are automatically i-measures, as shown by

\begin{pro}{w*hide}
Let $W$ be a linear subspace of $F^*$ such that $W$ is sequentially closed in the weak* topology
of $F^*$ and any weak* vector measure $\nu\dd \Mm \to F^*$ whose range is contained in $W$ is
a vector measure. Then any weak* i-measure $\mu\dd \Mm \to \llL(E,W) \subset \llL(E,F^*)$ is
an i-measure. In particular, if $W$ is a linear subspace of $F^*$ that is sequentially closed
in the weak* topology and contains no isomorphic copy of $\ell_{\infty}^{\RRR}$, then every
$\llL(E,W)$-valued weak* i-measure is an i-measure.
\end{pro}
\begin{proof}
Fix an infinite collection of pairwise disjoint sets $A_n \in \Mm$ and a bounded sequence of vectors
$x_n \in E$. For each $f \in F$ define $\nu_f\dd \Mm \to \KKK$ by $\nu_f(B) \df \sum_{n=1}^{\infty}
(\mu(A_n \cap B) x_n)(f)$. Since the set functions $\Mm \ni B \mapsto (\mu(A_n \cap B) x_n)(f) \in
\KKK$ are measures, we see (e.g. by the Vitali-Hahn-Saks-Nikodym theorem; consult Theorem~8
on page~23 in \cite{d-u}) that $\nu_f$ is a measure as well. Consequently, the formula $(\nu(B))(f)
\df \nu_f(B)\ (B \in \Mm,\ f \in F)$ correctly defines a weak* vector measure $\nu\dd \Mm \to F^*$.
What is more, it follows from the definition of $\nu$ and the property that $W$ is sequentially
closed in the weak* topology of $F$ that $\nu(B) \in W$ for any $B \in \Mm$. Thus, $\nu$ is a vector
measure, which implies that the series $\sum_{n=1}^{\infty} \nu(A_n)$ is convergent in the norm
topology. But $\nu(A_n) = \mu(A_n) x_n$ and consequently $\sum_{n=1}^{\infty} \mu(A_n)$ is
independently convergent. Since, in addition, $\scalar{f}{\mu(\bigcup_{n=1}^{\infty} A_n)(\cdot)} =
\sum_{n=1}^{\infty} \scalar{f}{\mu(A_n)(\cdot)}\ (f \in F)$, we see that $\mu(\bigcup_{n=1}^{\infty}
A_n) f = \sum_{n=1}^{\infty} \mu(A_n) f\ (f \in F)$ and we are done.\par
An additional claim follows from a celebrated result due to Diestel and Faires \cite{d-f} (see also
\cite{die} or Theorem~2 on page~20 in \cite{d-u}) which implies that each $W$-valued weak* vector
measure is a vector measure provided $W$ contains no isomorphic copy of $\ell_{\infty}^{\RRR}$.
\end{proof}

The proofs of the next two results are skipped. The first of them immediately follows from
the definition of the weak* integral for elements of $S_{\Mm}(Z,E)$, while the second is
a consequence of \THM{bw*c} and (BC*).

\begin{pro}{sc}
Let $W$ be a linear subspace of $F^*$ that is sequentially closed in the weak* topology of $F^*$.
If $\mu\dd \Mm \to \llL(E,W) \subset \llL(E,F^*)$ is a weak* i-measure \textup{(}where $\Mm$ is
a $\sigma$-algebra of subsets of a set $Z$\textup{)}, then $\int^{w*}_Z f \dint{\mu} \in W$ for any
$f \in M_{\Mm}(Z,E)$.
\end{pro}

\begin{pro}{w*T}
Let $\mu\dd \Mm(X) \to \llL(E,W) \subset \llL(E,F^*)$ be a weak* i-measure \textup{(}where $W$ is
a linear subspace of $F^*$ that is sequentially closed in the weak* topology of $F^*$\textup{)}. Let
$T\dd C(X,E) \to W$ be given by $T f \df \int^{w*}_X f \dint{\mu}$ and let $\bar{T}\dd M(X,E) \to
F^*$ be as specified in \COR{C-M}. Then $\|T\| = \|\mu\|_X$ and
\begin{equation*}
\bar{T} f = \int^{w*}_X f \dint{\mu} \qquad (f \in M(X,E)).
\end{equation*}
\end{pro}

\begin{thm}{w*}
Let $W$ be a linear subspace of $F^*$ that is sequentially closed in the weak* topology of $F^*$.
For every continuous linear operator $T\dd C(X,E) \to W$ there exists a unique weak* i-measure
$\mu\dd \Mm(X) \to \llL(E,W) \subset \llL(E,F^*)$ for which
\begin{equation}\label{eqn:w*repr}
T f = \int^{w*}_X f \dint{\mu} \qquad (f \in C(X,E)).
\end{equation}
Moreover, $\|T\| = \|\mu\|_X$.
\end{thm}
\begin{proof}
Assume $T\dd C(X,E) \to W$ is a continuous linear operator. The uniqueness of $\mu$ as well as
the additional claim of the theorem immediately follow from \PRO{w*T}. We shall now show
the existence of $\mu$. We infer from the proof of \PRO{vsc} that $T$ extends to $\bar{T}\dd M(X,E)
\to W$ which satisfies (BC*). We define $\mu\dd \Mm(X) \to \llL(E,W)$ by the rule $\mu(A) x \df
\bar{T}(j_A(\cdot) x)$ where $j_A\dd X \to \{0,1\}$ is the characteristic function of $A$ (here
we also continue the notational convention introduced in the proof of \LEM{1}). It is easily seen
that $\mu(A) \in \llL(E,W)$. Assume $A_n \in \Mm(X)$ are pairwise disjoint and let $x_n \in E$ be
uniformly bounded. Put $s_N \df \sum_{k=1}^N j_{A_k}(\cdot) x_k\ (N=1,2,\ldots,\infty)$. Notice that
the functions $s_n$ are uniformly bounded and converge pointwise (in the norm topology of $E$)
to $s_{\infty}$. So, it follows from (BC*) that the functionals $\sum_{k=1}^n \mu(A_k) x_k = \bar{T}
s_n$ converge to $\bar{T} s_{\infty}$ in the weak* topology of $F^*$. This implies that the series
$\sum_{k=1}^{\infty} \mu(A_k)$ is independently weak* convergent. What is more, if $x_k = x \in E$
for each $k$, then, under the above notations, $s_{\infty} = j_{\bigcup_{k=1}^{\infty} A_k}(\cdot)
x$ and we see that the series $\sum_{n=1}^{\infty} \mu(A_k) x$ converges in the weak* topology
of $F^*$ to $\bar{T} s_{\infty} = \mu(\bigcup_{k=1}^{\infty} A_k) x$. We conclude that $\mu$ is
a weak* i-measure. Finally, putting $L f \df \int^{w*}_X f \dint{\mu}$ for $f \in M(X,E)$, we see
that $L\dd M(X,E) \to F^*$ and $\bar{T}$ are two continuous functions which coincide
on $S_{\Mm(X)}(X,E)$. Since this last space is dense in $M(X,E)$, we conclude that $L = \bar{T}$ and
thus \eqref{eqn:w*repr} holds.
\end{proof}

The proof of the next result is omitted.

\begin{cor}{w*}
Let $F_{sc}$ be the smallest linear subspace of $F^{**}$ that contains $F$ and is sequentially
closed in the weak* topology of $F^{**}$. For every continuous linear operator $T\dd C(X,E) \to F$
there is a \textup{(}unique\textup{)} weak* i-measure $\mu\dd \Mm(X) \to \llL(E,F_{sc}) \subset
\llL(E,F^{**})$ for which \eqref{eqn:w*repr} holds.
\end{cor}

\begin{pro}{VSC}
Let $F$ be a vsc Banach space that contains no isomorphic copy of $\ell_{\infty}^{\RRR}$. For every
continuous linear operator $T\dd C(X,E) \to F$ there exists a unique i-measure $\mu\dd \Mm(X) \to
\llL(E,F)$ such that \eqref{eqn:vint} holds.
\end{pro}
\begin{proof}
We start from the existence part. There exists a linear isomorphism $\Phi\dd F \to W \subset Z^*$
such that $W$ is a linear subspace of a dual Banach space $Z^*$ that is sequentially closed
in the weak* topology (see \PRO{vsc-dual}). It follows from \THM{w*} that there is a weak* i-measure
$\nu\dd \Mm(X) \to \llL(E,W)$ such that
\begin{equation}\label{eqn:trm}
(\Phi \circ T) f = \int^{w*}_X f \dint{\nu} \qquad (f \in C(X,E)).
\end{equation}
Since $F$ contains no isomorphic copy of $\ell_{\infty}^{\RRR}$, so does $W$ and \PRO{w*hide}
implies that $\nu$ is an i-measure. We define $\mu\dd \Mm(X) \to \llL(E,F)$ by $\mu(A) \df \Phi^{-1}
\circ \nu(A)$. Straightforward calculations shows that $\mu$ is also an i-measure. Moreover, for
$u \in S_{\Mm(X)}(X,E)$ one simply has
\begin{equation}\label{eqn:tri}
\int_X u \dint{\mu} = \Phi^{-1}\Bigl(\int_X u \dint{\nu}\Bigr)
\end{equation}
and thus \eqref{eqn:tri} holds for all $u \in M(X,E)$. Consequently, \eqref{eqn:tri} and
\eqref{eqn:trm} yield \eqref{eqn:vint}.\par
To establish the uniqueness of $\mu$, it is enough to check that if $\lambda\dd \Mm(X) \to
\llL(E,F)$ is an i-measure such that $\int_X f \dint{\lambda} = 0$ for each $f \in C(X,E)$, then
$\lambda = 0$. But this simply follows from \THM{bwc} and the characterisation of $M(X,E)$ given
in \LEM{1}.
\end{proof}

\begin{rem}{marg}
\PRO{extend} and \LEM{norm} will show that, under the notation of \PRO{VSC}, $\|\mu\|_X = \|T\|$.
\end{rem}

Taking into account the characterisation of wsc Banach spaces formulated in \PRO{wsc}, the following
result is a little bit surprising.

\begin{cor}{vsc}
Let $F$ be a vsc Banach space that contains no isomorphic copy of $\ell_{\infty}^{\RRR}$. Every
continuous linear operator $T\dd C(X,E) \to F$ admits a unique linear extension $\bar{T}\dd M(X,E)
\to F$ such that \textup{(BC)} holds with $M(X,E)$ inserted in place of $\mmM(V)$. Moreover,
$\bar{T}$ is continuous and $\|\bar{T}\| = \|T\|$.
\end{cor}
\begin{proof}
Uniqueness, as usual, follows from \LEM{1} and (BC). To establish the existence, apply \PRO{VSC}
to get an i-measure $\mu$ such that \eqref{eqn:vint} holds and $\|\mu\|_X = \|T\|$ (see \REM{marg}).
Then define $\bar{T}\dd M(X,E) \to F$ by $\bar{T} f \df \int_X f \dint{\mu}$ and use \THM{bwc}
to show (BC).
\end{proof}

\begin{exm}{linfty}
Let us show that the assumption in \PRO{VSC} that $F$ contains no isomorphic copy
of $\ell_{\infty}^{\RRR}$ is essential. Put $F \df \ell_{\infty}^{\KKK}$. Since $F$ is a dual Banach
space, it is vsc. Now let $T\dd C([0,1],\KKK) \to F$ be given by $T f \df
(f(\frac1n))_{n=1}^{\infty}$. It is an elementary exercise to find a uniformly bounded sequence
of functions $f_n \in C([0,1],\KKK)$ which converge pointwise to $0$ but the vectors $T f_n$ diverge
in the norm topology. This is contradictory to \THM{bnc}, whose assertion has to be true for any
operator $T$ for which \PRO{VSC} holds.
\end{exm}

\begin{rem}{w*integrable}
As for i-measures, we wish to introduce the concept of integration of (possibly unbounded) functions
with respect to weak* i-measures. Let $\mu\dd \Mm \to \llL(E,F^*)$ be a weak* i-measure on a set $Z$
and $g\dd Z \to E$ an $\Mm$-measurable function (see \DEF{measurable}). Let $\Dd(g)$ be as specified
in \REM{integrable}. Define $\nu\dd \Dd(g) \to F^*$ by $\nu(A) \df \int^{w*}_Z j_A g \dint{\mu}$.
The function $g$ is said to be \textit{weak* integrable} if the set function $\nu\dd \Dd(g) \ni A
\mapsto \int^{w*}_Z j_A g \dint{\mu} \in F^*$ extends to a (necessarily unique) weak* vector measure
$\bar{\nu}\dd \Mm \to F^*$. If this happens, for each $A \in \Mm$ we define the \textit{weak*
integral} $\int^{w*}_A g \dint{\mu}$ (of $g$ on $A$ with respect to $\mu$) as $\bar{\nu}(A)$.
In the above situation, the set function $\nu$ is always a \textit{conditional} weak* vector
measure, which follows from \THM{bw*c}. Thus, every bounded $\Mm$-measurable function is weak*
integrable. Weak* integrable functions form a vector space and the weak* integral $\int_A$ (with
respect to $\mu$) is a linear operator (for each $A \in \Mm$).
\end{rem}

\section{Regularisation of i-measures}

In this section $Y = \Omega \sqcup \{\infty\}$ is a one-point compactification of $\Omega$.

\begin{dfn}{regular}
An i-measure $\mu$ defined on $\Bb(\Omega)$ is said to be \textit{regular} if every set $A \in
\Bb(\Omega)$ includes a $\sigma$-compact set $K$ such that $\mu$ vanishes on every Borel set
contained in $K \setminus A$.
\end{dfn}

It is an easy task to check that all regular i-measures form a linear subspace, to be denoted
by $\MmM_r(\Bb(\Omega),\llL(E,F))$, of $\MmM(\Bb(\Omega),\llL(E,F))$.\par
What we mean by a \textit{regularisation} of an i-measure is the property formulated below.

\begin{pro}{extend}
Every i-measure $\mu\dd \Mm(X) \to \llL(E,F)$ is uniquely extendable to a regular i-measure
$\bar{\mu}\dd \Bb(X) \to \llL(E,F)$. What is more, $\|\bar{\mu}\|_X = \|\mu\|_X$ and there exists
a regular measure $\bar{\lambda}\dd \Bb(X) \to [0,\infty)$ such that $\bar{\mu} \ll \bar{\lambda}$.
\end{pro}
\begin{proof}
Let $\lambda\dd \Mm(X) \to [0,\infty)$ be a measure such that $\mu \ll \lambda$ (see \COR{abs}).
Then $\lambda$ extends uniquely to a regular measure $\bar{\lambda}\dd \Bb(X) \to [0,\infty)$ (this
property may simply be concluded from the Riesz characterisation theorem applied for the linear
functional given by $C(X,\KKK) \ni f \mapsto \int_X f \dint{\lambda} \in \KKK$). The measure
$\bar{\lambda}$ has the following property:
\begin{itemize}
\item[($**$)] for any set $A \in \Bb(X)$ there exists a set $A^{\#} \in \Mm(X)$ such that
 $\bar{\lambda}(A \setminus A^{\#}) = \bar{\lambda}(A^{\#} \setminus A) = 0$.
\end{itemize}
Notice also that if $A$ and $A^{\#}$ are as specified above and $A^{\#\#} \in \Mm(X)$ is such that
$\bar{\lambda}(A \setminus A^{\#\#}) = \bar{\lambda}(A^{\#\#} \setminus A) = 0$, then $\mu(A^{\#}) =
\mu(A^{\#\#})$ (because $\lambda(A^{\#} \setminus A^{\#\#}) = \lambda(A^{\#\#} \setminus A^{\#}) =
0$ and $\mu \ll \lambda$). This observation means that the formula $\bar{\mu}(A) \df \mu(A^{\#})$
where, for $A \in \Bb(X)$, $A^{\#}$ is as specified in ($**$) correctly defines a set function
$\bar{\mu}\dd \Bb(X) \to \llL(E,F)$, which extends $\mu$. Now take a sequence of pairwise disjoint
sets $A_n \in \Bb(X)$. We can find a sequence of \textit{pairwise disjoint} sets $A_n^{\#} \in
\Mm(X)$ for which ($**$) is satisfied with $A_n$ and $A_n^{\#}$ inserted in place of $A$ and
$A^{\#}$ (respectively). Consequently, the series $\sum_{n=1}^{\infty} \mu(A_n)$ is independently
convergent, $\bar{\mu}(A_n) = \mu(A_n^{\#})$ for each $n$ and $\bar{\mu}(\bigcup_{n=1}^{\infty} A_n)
= \mu(\bigcup_{n=1}^{\infty} A_n^{\#})$, which implies that $\bar{\mu}$ is an i-measure and
$\|\bar{\mu}\|_X = \|\mu\|_X$.\par
Further, if $\bar{\lambda}(A) = 0$ and $A^{\#}$ is as specified in ($**$), then $\lambda(A^{\#}) =
0$ and, consequently, $\bar{\mu}(A) = \mu(A^{\#}) = 0$. This shows that $\bar{\mu} \ll
\bar{\lambda}$ (see \COR{abs-abs}). Finally, for any $A \in \Bb(X)$ one can find a $\sigma$-compact
set $K \subset A$ such that $\bar{\lambda}(A \setminus K) = 0$ and hence $\bar{\mu}$ vanishes
on each Borel subset of $A \setminus K$. This proves that $\bar{\mu}$ is regular.\par
To establish the uniqueness of $\bar{\mu}$, assume $\bar{\mu}'\dd \Bb(X) \to \llL(E,F)$ is another
regular i-measure extending $\mu$. For each $e \in E$ and $\psi \in F^*$, we define
$\bar{\mu}_{e,\psi}\dd \Bb(X) \to \KKK$ (and similarly $\bar{\mu}'_{e,\psi}\dd \Bb(X) \to \KKK$)
by $\bar{\mu}_{e,\psi}(A) \df (\psi \circ \bar{\mu}(A))(e)$. It follows from the regularity
of $\bar{\mu}$ and $\bar{\mu}'$ that $\bar{\mu}_{e,\psi}$ and $\bar{\mu}'_{e,\psi}$ are regular
scalar-valued measures. But both these scalar-valued measures coincide on $\Mm(X)$ and hence
$\bar{\mu}'_{e,\psi} = \bar{\mu}_{e,\psi}$ (thanks to the Riesz characterisation theorem).
Consequently, $\bar{\mu}' = \bar{\mu}$.
\end{proof}

As in \PRO{extend}, we denote by $\bar{\mu}$ the unique regular Borel i-measure which extends
an i-measure $\mu$ defined on $\Mm(X)$.

\begin{cor}{reg}
For an i-measure $\mu\dd \Bb(X) \to \llL(E,F)$, \tfcae
\begin{enumerate}[\upshape(i)]
\item $\mu$ is regular;
\item for any $e \in E$ and $\psi \in F^*$, the scalar-valued measure $\mu_{e,\psi}\dd \Bb(X) \ni A
 \mapsto (\psi \circ \mu(A))(e) \in \KKK$ is regular;
\item there exists a regular measure $\rho\dd \Bb(X) \to [0,\infty)$ such that $\mu \ll \rho$;
\item there exists a regular measure $\lambda\dd \Bb(X) \to [0,\infty)$ such that $\mu \ll \lambda$
 and $\lambda(A) \leqsl \|\mu\|_A$ for each $A \in \Bb(X)$.
\end{enumerate}
\end{cor}
\begin{proof}
Implications (iv)$\implies$(iii)$\implies$(i)$\implies$(ii) are clear. Further, it follows from
\PRO{extend} applied for the i-measure $\mu\bigr|_{\Mm(X)}$ that (iii) follows from (i), and that
(i) is implied by (ii) (see the proof of the uniqueness part in \PRO{extend}). So, it remains
to check that (iii) implies (iv). Let $\rho$ and $\lambda$ be as specified in (iii) and \COR{abs},
respectively. We may assume $\lambda$ satisfies \eqref{eqn:mut-abs}. It remains to check that
$\lambda$ is regular, which simply follows from the fact that $\lambda \ll \rho$ (because $\lambda$
vanishes precisely on those sets on which $\mu$ vanishes---see the proof of \COR{abs-abs}).
\end{proof}

The proof of the next (very simple) result is left to the reader.

\begin{lem}{one-point}
An i-measure $\mu\dd \Bb(Y) \to \llL(E,F)$ is regular iff $\nu \df \mu\bigr|_{\Bb(\Omega)}$, treated
as an i-measure on $\Omega$, is regular.
\end{lem}

\begin{lem}{norm}
Let $\mu\dd \Bb(\Omega) \to \llL(E,F)$ be a regular i-measure. Then $\|T_{\nu}\| = \|\nu\|_{\Omega}$
where
\begin{equation}\label{eqn:Tnu}
T_{\nu}\dd C_0(\Omega,E) \ni f \mapsto \int_{\Omega} f \dint{\nu} \in F.
\end{equation}
\end{lem}
\begin{proof}
It is clear that $\|T_{\nu}\| \leqsl \|\nu\|_{\Omega}$. To show the reverse inequality, take
a finite collection of $N$ pairwise disjoint sets $A_k \in \Bb(\Omega)$ and a corresponding system
of $N$ vectors $x_k \in E$ whose norms do not exceed $1$. We only need to check that $\|\sum_{k=1}^N
\nu(A_k) x_k\| \leqsl \|T_{\nu}\|$. It follows from the definition of a regular i-measure that for
each $k$ there exists a sequence of compact subsets $K_n^{(k)}$ of $A_k$ such that
$\lim_{n\to\infty} \|\nu(K_n^{(k)}) - \nu(A_k)\| = 0$. Then, when $n$ is fixed, the sets $K_n^{(k)}$
are pairwise disjoint; and $\lim_{n\to\infty} \|\sum_{k=1}^N \nu(K_n^{(k)}) x_k\| = \|\sum_{k=1}^N
\nu(A_k) x_k\|$. This argument allows us to assume the sets $A_k$ are compact. Further, we conclude
(again) from the regularity of $\nu$ that for each $k$ there is a decreasing sequence of open
supersets $U_n^{(k)}$ of $A_k$ such that $\nu$ vanishes on every Borel subset
of $\bigcap_{n=1}^{\infty} U_n^{(k)} \setminus A_k$. We may also assume that, in addition, the sets
$U_1^{(k)}$ are pairwise disjoint. Now, using e.g. Urysohn's lemma, (for each $k$) we may find
a decreasing sequence of compact $\ggG_{\delta}$-sets $F_n^{(k)}$ with $A_k \subset F_n^{(k)}
\subset U_n^{(k)}$. Then, for each fixed $n$, the sets $F_n^{(k)}$ are pairwise disjoint;
$\lim_{n\to\infty} \|\nu(F_n^{(k)}) - \nu(\bigcap_{n=1}^{\infty} F_n^{(k)})\| = 0$ and
$\nu(\bigcap_{n=1}^{\infty} F_n^{(k)}) = \nu(A_k)$ (because $\bigcap_{n=1}^{\infty} F_n^{(k)}
\setminus A_k \subset \bigcap_{n=1}^{\infty} U_n^{(k)} \setminus A_k$). Hence, arguing as before,
we may and do assume the sets $A_k$ are (compact and) $\ggG_{\delta}$. Take pairwise disjoint open
sets $V_k$ such that $A_k \subset V_k$. Since $A_k$ is $\ggG_{\delta}$ and compact, there exists
a sequence of continuous functions $u_n^{(k)}\dd \Omega \to [0,1]$ which converge pointwise
(as $n \to \infty$) to the characteristic function $j_{A_k}$ of $A_k$ and vanish off $V_k$. Put
\begin{equation*}
f_n \df \sum_{k=1}^N u_n^{(k)}(\cdot) x_k
\end{equation*}
and observe that $f_n \in C_0(\Omega,E)$, $\|f_n\| \leqsl 1$ (since $\|x_k\| \leqsl 1$ for all $k$
and the sets $V_k$ are pairwise disjoint) and the functions $f_n$ converge pointwise (in the norm
topology of $E$) to $\sum_{k=1}^N j_{A_k}(\cdot) x_k$. So, $\|T_{\nu} f_n\| \leqsl \|T_{\nu}\|$ for
each $n$; and an application of \THM{bnc} gives $\|\sum_{k=1}^N \nu(A_k) x_k\| = \|\int_{\Omega}
\sum_{k=1}^N j_{A_k}(\omega) x_k \dint{\mu(\omega)}\| = \lim_{n\to\infty} \|T_{\nu} f_n\| \leqsl
\|T_{\nu}\|$.
\end{proof}

\begin{thm}{vsc}
Let $F$ be a vsc Banach space that contains no isomorphic copy of $\ell_{\infty}^{\RRR}$ and
$\Omega$ be a locally compact Hausdorff space. For every continuous linear operator $T\dd
C_0(\Omega,E) \to F$ there exists a unique regular Borel i-measure $\mu\dd \Bb(\Omega) \to
\llL(E,F)$ such that
\begin{equation*}
T f = \int_{\Omega} f \dint{\mu} \qquad (f \in C_0(\Omega,E)).
\end{equation*}
Moreover, $\|T\| = \|\mu\|_{\Omega}$.
\end{thm}
\begin{proof}
Below we shall continue to denote by $T_{\nu}$ the operator defined by \eqref{eqn:Tnu} (provided
$\nu \in \MmM_r(\Bb(\Omega),\llL(E,F))$).\par
For each $e \in F$, let $c_e\dd \Omega \to F$ stand for the constant function whose only value is
$e$. Define $S\dd C(Y,E) \to F$ by $S u \df T(u\bigr|_{\Omega} - c_{u(\infty)})$. It is clear that
$S$ is continuous and linear. So, it follows from \PRO{VSC} that there is an i-measure $\nu\dd
\Mm(Y) \to \llL(E,F)$ for which $S u = \int_Y u \dint{\nu}\ (u \in C(Y,E))$. We define $\mu\dd
\Bb(\Omega) \to \llL(E,F)$ as the restriction of $\bar{\nu}$ to $\Bb(\Omega)$. We conclude from
\LEM{one-point} that $\mu$ is regular. Since every function $g \in C_0(\Omega,E)$ extends
to a continuous function $\bar{g}$ on $Y$ which vanish at $\infty$, we see that $\int_{\Omega} g
\dint{\mu} = \int_Y \bar{g} \dint{\bar{\nu}}$. But $\int_Y \bar{g} \dint{\bar{\nu}} = \int_Y \bar{g}
\dint{\nu} = S \bar{g} = T(g)$ and hence $T = T_{\mu}$.\par
Finally, since the operator $\Phi\dd \MmM_r(\Bb(\Omega),\llL(E,F)) \ni \nu \mapsto T_{\nu} \in
\llL(C_0(\Omega,E),F)$ is linear, \LEM{norm} yields that $\Phi$ is isometric and hence one-to-one,
which finishes the proof.
\end{proof}

\begin{proof}[Proof of \THM{3}] Just notice that all wsc Banach spaces as well as all dual Banach
spaces are vsc and all wsc Banach spaces contain no isomorphic copy of $\ell_{\infty}^{\RRR}$ (since
they even contain no isomorphic copy of $c_0$), and then apply \THM{vsc} and \LEM{norm}.
\end{proof}

\begin{cor}{VSC}
Assume $F$ is a vsc Banach space that contains no isomorphic copy of $\ell_{\infty}^{\RRR}$ and
$T\dd C_0(\Omega,E) \to F$ is continuous and linear. If $f_n \in C_0(\Omega,E)$ are uniformly
bounded and converge pointwise \textup{(}to a possibly discontinuous function\textup{)} in the norm
topology of $E$, then $T f_n$ converge in the norm topology of $F$. In particular, if, in addition,
$E = \KKK$, then $T$ sends weakly fundamental sequences into norm convergent sequences.
\end{cor}
\begin{proof}
It follows from \THM{vsc} that $T f = \int_{\Omega} f \dint{\mu}$ for some regular Borel i-measure
$\mu$. So, the first assertion follows from \THM{bnc}. The additional claim follows from the first
and the characterisation of weakly fundamental sequences in $C_0(\Omega,\KKK)$ (these are precisely
those which are uniformly bounded and converge pointwise to a possibly discontinuous function).
\end{proof}

The reader interested in other results on continuous linear operators defined on the spaces
of the form $C(X,\KKK)$ (into arbitrary Banach spaces) is referred to Chapter~VI in \cite{d-u}.

\begin{exm}{more}
Taking into account all properties established above, a natural question arises whether the first
assertion of \COR{VSC} holds for more general cases, such as:
\begin{itemize}
\item $T\dd C(X,E) \to F$ where $F$ is an arbitrary Banach space that contains no isomorphic copy
 of $\ell_{\infty}^{\RRR}$;
\item $T\dd V \to F$ where $F$ is wsc and $V$ is a linear subspace of $C(X,E)$
\end{itemize}
(above $T$ is assumed to be continuous and linear). Let us briefly explain that, in general,
the answer is negative (in both the above cases). For a counterexample in the first settings, just
put $F \df C([0,1],\KKK)$ and take the identity operator on $F$. To disprove the assertion
of \COR{VSC} in the second case, take an isometric copy $V$ of $F \df L^2([0,1])$ in $C([0,1],\KKK)$
and define $T$ as a linear isometry of $V$ onto $L^2([0,1])$.
\end{exm}

\COR{VSC} enables us to give an example of classical Banach spaces which are not vsc.

\begin{cor}{notvsc}
For every infinite second countable locally compact topological space $\Omega$, the Banach space
$C_0(\Omega,\KKK)$ is not vsc. In particular, $c_0$ and $C([0,1],\KKK)$ are not vsc.
\end{cor}
\begin{proof}
Since $F \df C_0(\Omega,\KKK)$ is separable, it contains no isomorphic copy
of $\ell_{\infty}^{\RRR}$. So, if $F$ was vsc, the identity operator on $F$ would satisfy
the assertion of \COR{VSC}, which is false.
\end{proof}

We now turn to regular weak* i-measures.

\begin{dfn}{w*reg}
A weak* i-measure $\mu\dd \Bb(\Omega) \to \llL(E,F^*)$ is \textit{regular} if for any $f \in F$,
the i-measure $\mu_f\dd \Bb(\Omega) \ni A \mapsto \scalar{f}{\mu(A)(\cdot)} \in \llL(E,\KKK)$ is
regular.
\end{dfn}

The reader should notice that the set of all $\llL(E,F^*)$-valued regular Borel weak* i-measures
on $\Omega$ is a vector space. We also wish to emphasize that, in general, for a weak* i-measure
$\mu$ and a Borel set $A$ they may be no $\sigma$-compact subset $K$ of $A$ such that $\mu$ vanishes
on every Borel subset of $A \setminus K$.

\begin{lem}{reg}
Every weak* i-measure $\mu\dd \Mm(X) \to \llL(E,F^*)$ extends to a unique regular weak* i-measure
$\bar{\mu}\dd \Bb(X) \to \llL(E,F^*)$. Moreover, $\|\bar{\mu}\|_X = \|\mu\|_X$.
\end{lem}
\begin{proof}
For each $f \in F$, let $\nu_f\dd \Bb(X) \to \llL(E,\KKK)$ be the unique regular i-measure which
extends $\mu_f\dd \Mm(X) \ni A \mapsto \scalar{f}{\mu(A)(\cdot)} \in \llL(E,\KKK)$ (see
\PRO{extend}). It follows from the uniqueness of the extension that the operator $F \ni f \mapsto
\nu_f \in \MmM_r(\Bb(X),\llL(E,\KKK))$ is linear. Moreover, $\|\nu_f(A)\| \leqsl \|\nu_f\|_X \cdot
\|f\| = \|\mu_f\|_X \cdot \|f\| \leqsl \|\mu\|_X \cdot \|f\|$. One concludes that the rule
$\scalar{f}{\bar{\mu}(A)(\cdot)} = \nu_f(A)\ (f \in F,\ A \in \Bb(X))$ correctly defines a set
function $\bar{\mu}\dd \Bb(X) \to \llL(E,F^*)$. It follows from the very definition of $\bar{\mu}$
that $\bar{\mu}$ is a regular weak* i-measure. What is more, if $A_k \in \Bb(X)$ are paiwise
disjoint and $x_k \in E$ have norms not exceeding $1$, then
\begin{align*}
\Bigl\|\sum_{k=1}^N \bar{\mu}(A_k) x_k\Bigr\| &=
\sup \Bigl\{\Bigl|\sum_{k=1}^N (\bar{\mu}(A_k) x_k)(f)\Bigr|\dd\ f \in F,\ \|f\| \leqsl 1\Bigr\}\\
&= \sup \Bigl\{\Bigl|\sum_{k=1}^N \nu_f(A_k) x_k\Bigr|\dd\ f \in F,\ \|f\| \leqsl 1\Bigr\}\\
&= \sup \{\|\nu_f\|_X\dd\ f \in F,\ \|f\| \leqsl 1\} \leqsl \|\mu\|_X
\end{align*}
and therefore $\|\bar{\mu}\|_X = \|\mu\|_X$. The uniqueness of $\bar{\mu}$ follows from
\PRO{extend}.
\end{proof}

As for i-measures, for any weak* i-measure $\mu\dd \Mm(X) \to \llL(E,F^*)$, we shall denote
by $\bar{\mu}\dd \Bb(X) \to \llL(E,F^*)$ the unique extension of $\mu$ to a regular weak* i-measure.
It is worth noting here that if $W$ is a linear subspace of $F^*$ that is sequentially closed
in the weak* topology and $\mu(\Mm(X)) \subset \llL(E,W)$, then, in general, the range
of $\bar{\mu}$ may contain operators which do not belong to $\llL(E,W)$. This is why we deal here
with dual Banach spaces instead of their weak* sequentially closed subspaces.

As for i-measures, we have

\begin{lem}{one-point*}
A weak* i-measure $\mu\dd \Bb(Y) \to \llL(E,F^*)$ is regular iff $\nu \df \mu\bigr|_{\Bb(\Omega)}$,
treated as a weak* i-measure on $\Omega$, is regular.
\end{lem}
\begin{proof}
The assertion immediately follows from \LEM{one-point}.
\end{proof}

\begin{lem}{norm*}
For every regular weak* i-measure $\mu\dd \Bb(\Omega) \to \llL(E,F^*)$, $\|T_{\mu}\| =
\|\mu\|_{\Omega}$ where
\begin{equation}\label{eqn:Tmu*}
T_{\mu}\dd C_0(\Omega,E) \ni u \mapsto \int^{w*}_{\Omega} u \dint{\mu} \in F^*.
\end{equation}
\end{lem}
\begin{proof}
As usual, for each $f \in F$, denote by $\mu_f\dd \Bb(\Omega) \to \llL(E,\KKK)$ a regular i-measure
given by $\mu_f(A) = \scalar{f}{\mu(A)(\cdot)}$. Observe that $(T_{\mu} u)(f) = \int_{\Omega} u
\dint{\mu_f}$ for any $f \in F$ and $u \in C_0(\Omega,E)$. It follows from \LEM{norm} that
$\|\scalar{f}{T_{\mu}(\cdot)}\| = \|\mu_f\|_{\Omega}$ and therefore $\|T_{\mu}\| =
\sup\{\|\mu_f\|_{\Omega}\dd\ f \in F,\ \|f\| \leqsl 1\} = \|\mu\|_{\Omega}$.
\end{proof}

\begin{thm}{W*}
For every continuous linear operator $T\dd C_0(\Omega,E) \to F^*$ there exists a unique regular
weak* i-measure $\mu\dd \Bb(\Omega) \to \llL(E,F^*)$ such that $T = T_{\mu}$ where $T_{\mu}$ is
given by \eqref{eqn:Tmu*}. Moreover, $\|T\| = \|\mu\|_{\Omega}$.
\end{thm}
\begin{proof}
Thanks to \LEM{norm*}, it suffices to show the existence of $\mu$ (see the last paragraph
in the proof of \THM{vsc}). We repeat some of arguments used in the proof of \THM{vsc}. For each
$e \in E$, let $c_e\dd \Omega \to E$ be the constant function whose only value is $e$. Define $S\dd
C(Y,E) \to F^*$ by $S u \df T(u\bigr|_{\Omega} - c_{u(\infty)})$. It follows from \THM{w*} that
there exists an i-measure $\nu\dd \Mm(X) \to \llL(E,F^*)$ such that $S u = \int^{w*}_Y u \dint{\nu}$
for all $u \in C(Y,E)$. We define $\mu$ as the restriction of $\bar{\nu}$ (see \LEM{reg})
to $\Bb(\Omega)$. We infer from \LEM{one-point*} that $\mu$ is a regular weak* i-measure. Now
it suffices to repeat the reasoning presented in \THM{vsc} in order to verify that $T = T_{\mu}$.
\end{proof}

We conclude the section with the following consequence of \THM{W*}, whose proof is left
to the reader.

\begin{cor}[General Riesz Characterisation Theorem]{Riesz}
For any continuous linear operator $T\dd C_0(\Omega,E) \to F$ there exists a unique regular weak*
i-measure $\mu\dd \Bb(\Omega) \to \llL(E,F^{**})$ such that $T f = \int^{w*}_{\Omega} f \dint{\mu}$
for any $f \in C_0(\Omega,E)$. Moreover, $\|T\| = \|\mu\|_{\Omega}$.
\end{cor}

\section{Closure of a convex set}

As we shall see, \THM{2} is a consequence of the next result. For the need of its formulation,
we introduce the following

\begin{dfn}{bar(M)}
Let $D$ be a Borel subset of $\Omega$. For any set $A \subset M_{\Bb(D)}(D,E)$, the space
$\bar{\mmM}(A)$ is defined as the smallest set among all $B \subset M_{\Bb(D)}(D,E)$ such that:
\begin{enumerate}[($\bar{\textup{M}}$1)]\addtocounter{enumi}{-1}
\item $A \subset B$;
\item a function $u \in M_{\Bb(D)}(D,E)$ belongs to $B$ provided the following condition
 is fulfilled:
 \begin{itemize}
 \item[(aec)] for every finite regular Borel measure $\mu$ on $D$ there exist a uniformly bounded
  sequence of functions $u_n \in B$ and a set $Z \in \Bb(D)$ with $\mu(Z) = 0$ such that the vectors
  $u_n(\omega)$ converge to $u(\omega)$ in the weak topology of $E$ for any $\omega \in D \setminus
  Z$.
 \end{itemize}
\end{enumerate}
It is an easy exercise that $\mmM(A) \subset \bar{\mmM}(A)$ for any $A \subset M_{\Bb(D)}(D,E)$, and
that $\bar{\mmM}(V)$ is a linear subspace of $M_{\Bb(D)}(D,E)$ provided $V$ is so.\par
Using \LEM{1}, one may check that $\bar{\mmM}(C(X,E)) = M_{\Bb(X)}(X,E)$ for any compact space $X$.
\end{dfn}

\begin{thm}{closure}
Let $\kkK$ be a convex subset of $C_0(\Omega,E)$ and $\bbB$ be a countable collection of pairwise
disjoint Borel subsets of $\Omega$ that cover $\Omega$. For a function $f \in C_0(\Omega,E)$ \tfcae
\begin{enumerate}[\upshape(i)]
\item $f$ belongs to the norm closure \textup{(}in $C_0(\Omega,E)$\textup{)} of $\kkK$;
\item $f\bigr|_S \in \bar{\mmM}\bigl(\kkK\bigr|_S\bigr)$ \textup{(}where $\kkK\bigr|_S \df
 \bigl\{g\bigr|_S \in C(S,E)\dd\ g \in \kkK\bigr\}$\textup{)} for every Borel set $S \subset \Omega$
 such that $S \cap B$ is $\sigma$-compact for each $B \in \bbB$;
\item there exists a real constant $R > 0$ such that $f\bigr|_L \in
 \bar{\mmM}\bigl((\kkK \cap B(R))\bigr|_L\bigr)$ \textup{(}where $B(R) \df \{g \in C_0(\Omega,E)\dd\
 \|g\| \leqsl R\}$\textup{)} for every $L \in \Bb(\Omega)$ such that the set $L \cap B$ is compact
 for each $B \in \bbB$ and nonempty only for a finite number of such $B$.
\end{enumerate}
\end{thm}
\begin{proof}
We may and do assume that $\kkK$ is nonempty. It is readily seen that both conditions (ii) and (iii)
are implied by (i). First we shall show that (i) follows from (ii). Assume $f$ satisfies (ii) and
suppose, on the contrary, that $f$ is not in the norm closure of $\kkK$. We infer from
the separation theorem that there is a continuous linear functional $\psi\dd C_0(\Omega,E) \to \KKK$
such that $\gamma \df \sup\{\RE(\psi_0(u))\dd\ u \in \kkK\} < \RE(\psi(f))$. Since $\KKK$ is wsc,
it follows from \THM{3} that $\psi$ is of the form
\begin{equation*}
\psi(g) = \int_{\Omega} g \dint{\mu} \qquad (g \in C_0(\Omega,E))
\end{equation*}
for some $\llL(E,\KKK)$-valued regular Borel i-measure $\mu$. Further, we infer from the regularity
of $\mu$ that for any $B \in \bbB$ there is a $\sigma$-compact set $S_B \subset B$ such that $\mu$
vanishes on every Borel subset of $B \setminus S_B$. We put $S \df \bigcup_{B\in\bbB} S_B$. Since
$\bbB$ is countable, we see that $S \in \Bb(\Omega)$. What is more, for each $B \in \bbB$, $S \cap B
= S_B$ (because members of $\bbB$ are pairwise disjoint) and thus $S \cap B$ is $\sigma$-compact.
For any function $u \in M_{\Bb(S)}(S,E)$ we shall denote by $u^{\#}$ the (unique) extension of $u$
to a member of $M_{\Bb(\Omega)}(\Omega,E)$ which vanishes off $S$. We shall now verify that
$f\bigr|_S \notin \bar{\mmM}\bigl(\kkK\bigr|_S\bigr)$ (which contradicts (ii)). To this end, it is
enough to show that
\begin{equation}\label{eqn:gamma}
\RE\Bigl(\int_{\Omega} u^{\#} \dint{\mu}\Bigr) \leqsl \gamma
\end{equation}
for any $u \in \bar{\mmM}\bigl(\kkK\bigr|_S\bigr)$. To do that, denote by $\hhH$ the set of all
functions $u \in M_{\Bb(S)}(S,E)$ for which \eqref{eqn:gamma} holds. Since $\mu$ vanishes on every
Borel subset of $\Omega \setminus S$, we see that $\kkK\bigr|_S \subset \hhH$. Now assume a function
$u \in M_{\Bb(S)}(S,E)$ satisfies condition (aec) (with $D \df S$ and $B \df \hhH$). Taking into
account \COR{reg}, we conclude that there are a uniformly bounded sequence of functions $u_n \in
\hhH$ and a set $Z \in \Bb(S)$ such that $\mu$ vanishes on every Borel subset of $S \setminus Z$ and
the vectors $u_n(\omega)$ converge to $u(\omega)$ in the weak topology of $E$ for any $\omega \in S
\setminus Z$. One easily infers from \THM{bwc} that then $\lim_{n\to\infty} \int_{\Omega} u_n^{\#}
\dint{\mu} = \int_{\Omega} u^{\#} \dint{\mu}$ and therefore the set $B \df \hhH$ satisfies condition
($\bar{\textup{M}}$1). Consequently, $\bar{\mmM}\bigl(\kkK\bigr|_S\bigr) \subset \hhH$ and we are
done.\par
We now turn to the proof that (i) is implied by (iii). This part is more subtle. Let $R > 0$ be
as specified in (iii). We shall show that $f$ belongs to the norm closure of $\kkK \cap B(R)$.
To this end, replacing $\kkK$ by $\kkK \cap B(R)$, we may assume that $\kkK \subset B(R)$ is such
that
\begin{itemize}
\item[(iii')] $f\bigr|_L \in \bar{\mmM}\bigl(\kkK\bigr|_L\bigr)$ for every $L \in \Bb(\Omega)$ such
 that the set $L \cap B$ is compact for each $B \in \bbB$ and nonempty only for a finite number
 of such $B$.
\end{itemize}
Enlarging, if necessary, $R$, we may and do assume that $f \in B(R)$ as well. As before, we suppose,
on the contrary, that $f$ is not in the norm closure of $\kkK$ and take an $\llL(E,\KKK)$-valued
regular Borel i-measure $\mu$ such that
\begin{equation}\label{eqn:gam}
\RE\Bigl(\int_{\Omega} u \dint{\mu}\Bigr) \leqsl \gamma
\end{equation}
for all $u \in \kkK$ and some real constant $\gamma$, while
\begin{equation}\label{eqn:f-gam}
(\epsi \df\,)\ \frac13 \Bigl(\RE\Bigl(\int_{\Omega} f \dint{\mu}\Bigr) - \gamma\Bigr) > 0.
\end{equation}
Further, let $\lambda$ be a finite nonnegative regular Borel measure on $\Omega$ for which $\mu \ll
\lambda$. Using the last property, take $\delta > 0$ such that $\|\mu\|_A \leqsl \frac{\epsi}{R}$
whenever $A \in \Bb(\Omega)$ is such that $\lambda(A) \leqsl 2 \delta$. Write $\bbB =
\{B_1,B_2,\ldots\}$ and for any $n > 0$ take a compact set $L_n \subset B_n$ for which $\lambda(B_n
\setminus L_n) \leqsl \frac{\delta}{2^n}$. Further, let $N > 0$ be such that $\sum_{n=N+1}^{\infty}
\lambda(B_n) \leqsl \delta$. We put $L \df \bigcup_{n=1}^N L_n\ (\in \Bb(\Omega))$. We see that
$L \cap B_n$ coincides with $L_n$ for $n \leqsl N$ and is empty otherwise. Our aim is to show that
$f\bigr|_L \notin \bar{\mmM}\bigl(\kkK\bigr|_L\bigr)$. Observe that $\lambda(\Omega \setminus L)
\leqsl 2 \delta$ and therefore $\|\mu\|_{\Omega \setminus L} \leqsl \epsi/R$. Consequently,
$|\int_{\Omega} j_{\Omega \setminus L} u \dint{\mu}| \leqsl \epsi$ whenever $u \in B(R)$ (where,
as usual, $j_{\Omega \setminus L}$ denotes the characteristic function of $\Omega \setminus L$). So,
we conclude from \eqref{eqn:gam} and \eqref{eqn:f-gam} that
\begin{equation}\label{eqn:gam2}
\RE\Bigl(\int_{\Omega} u^{\#} \dint{\mu}\Bigr) \leqsl \gamma + \epsi
\end{equation}
for all $u \in \kkK\bigr|_L$ and $\RE(\int_{\Omega} j_L f \dint{\mu}) > \gamma + \epsi$. Now
similarly as in the proof that (i) follows from (ii), one shows that \eqref{eqn:gam2} holds for all
$u \in \bar{\mmM}\bigl(\kkK\bigr|_L\bigr)$ and hence $f\bigr|_L \notin
\bar{\mmM}\bigl(\kkK\bigr|_L\bigr)$ (because $(f\bigr|_L)^{\#} = j_L f$).
\end{proof}

\begin{cor}{bd}
Let $\kkK$ be a convex set in $C_0(\Omega,E)$.
\begin{enumerate}[\upshape(a)]
\item If $\kkK$ is bounded, its norm closure consists precisely of those functions $f \in
 C_0(\Omega,E)$ that $f\bigr|_L \in \bar{\mmM}\bigl(\kkK\bigr|_L\bigr)$ for any compact set $L
 \subset \Omega$.
\item If $\Omega$ is compact, the norm closure of $\kkK$ coincides with $\bar{\mmM}(\kkK)$.
\end{enumerate}
\end{cor}
\begin{proof}
In both the cases put $\bbB \df \{\Omega\}$. In case (a), take $R > 0$ such that $\kkK \subset B(R)$
and apply item (iii) of \THM{closure}. In case (b) just apply point (ii) of that result.
\end{proof}

\begin{pro}{C*}
Let $\Aa$ be a $C^*$-algebra and $\aaA$ be a $*$-subalgebra of $C_0(\Omega,\Aa)$. Let $\bbB$ be
a countable collection of pairwise disjoint Borel subsets of $\Omega$ that cover $\Omega$. The norm
closure of $\aaA$ consists precisely of those functions $f \in C_0(\Omega,\Aa)$ such that
\begin{itemize}
\item[(cc)] $f\bigr|_K \in \bar{\mmM}\bigl(\aaA\bigr|_K\bigr)$ for every set $K \in \Bb(\Omega)$
 such that the set $K \cap B$ is compact for each $B \in \bbB$ and nonempty only for a finite number
 of such $B$.
\end{itemize}
In particular, if \textup{(}$f \in C_0(\Omega,\Aa)$ and\textup{)} $f\bigr|_L$ belongs
to $\bar{\mmM}\bigl(\aaA\bigr|_L\bigr)$ for any compact set $L \subset \Omega$, then $f$ is
in the uniform closure of $\aaA$.
\end{pro}
\begin{proof}
First of all, we may and do assume that $\aaA$ is closed. It is enough to check that every function
$f \in C_0(\Omega,\Aa)$ for which (cc) holds belongs to the norm closure of $\aaA$. To this end,
take $R > \|f\|$. We shall show that condition (iii) of \THM{closure} (with $\kkK = \aaA$) holds for
such $R$ (which will finish the proof). Let $L \subset \Omega$ be as specified in that condition
(or, equivalently, as specified in (cc)). It follows from (cc) that $f\bigr|_L \in
\bar{\mmM}\bigl(\aaA\bigr|_L\bigr)$. Now point (b) of \COR{bd} (applied for $\Omega \df L$ and $\kkK
\df \aaA\bigr|_L$) yields that $f\bigr|_L$ belongs to the norm closure of $\aaA\bigr|_L$. Since
the function $\aaA \ni g \mapsto g\bigr|_L \in C_0(L,E)$ is a $*$-homomorphism (with range
$\aaA\bigr|_L$) between $C^*$-algebras, it sends the open unit ball of $\aaA$ onto the open unit
ball of $\aaA\bigr|_L$. Consequently, $f\bigr|_L \in (\aaA \cap B(R))\bigr|_L$ and we are done.
\end{proof}

\begin{proof}[Proof of \THM{2}]
Each of the three cases is a special case of one of \COR{bd} and \PRO{C*}.
\end{proof}

\THM{closure} is at least surprising and seems to be a convenient tool. Recently we use its
consequence---\PRO{C*} (in its almost exact form)---to describe models for all so-called
\textit{subhomogeneous} $C^*$-algebras (which may be seen as a solution of a long-standing problem).
The paper on this is in preparation. Below we give an illustrative example of usefulness of \THM{2}.
(The result below is certainly known.)

\begin{cor}{0,1}
Let $d$ denote the natural metric on $X \df [0,1]$. The linear span $V$ of all functions
$d(x,\cdot)$ is dense in $C(X,\RRR)$.
\end{cor}
\begin{proof}
Thanks to \THM{2}, it suffices to show that $\mmM(V)$ contains all continuous functions, which
is quite easy: $d(0,\cdot) + d(1,\cdot) \equiv 1$ and for any $x \in X \setminus \{1\}$ and small
enough $h > 0$ the functions $\frac1h (d(x+h,\cdot) - d(x,\cdot))$ are uniformly bounded, belong
to $V$ and converge pointwise to the function given by
\begin{equation*}
t \mapsto \begin{cases}1, & t \leqsl x\\-1, & t > x.\end{cases}
\end{equation*}
We conclude that the characteristic function of $[0,x]$ is a member of $\mmM(V)$ for any $x \in X$.
So, the characteristic functions of all intervals of the form $(a,b]$ (with $0 \leqsl a < b \leqsl
1$) also belong to $V$. Noticing that every continuous function on $X$ is a uniform limit of linear
combinations of such functions, we finish the proof.
\end{proof}

\begin{rem}{specific}
The assertion of \COR{0,1} (under the notations of that result) is equivalent to the following
property:
\begin{quote}
If two complex-valued Borel measures $\mu$ and $\nu$ on $X$ satisfy
\begin{equation}\label{eqn:id}
\int_X d(x,t) \dint{\mu(t)} = \int_X d(y,t) \dint{\nu(t)} \quad \textup{for all } x \in X,
\end{equation}
then $\mu = \nu$.
\end{quote}
We leave it as an exercise that there exists a finite metric space $(X,d)$ such that \eqref{eqn:id}
holds for two \textit{different} probabilistic measures $\mu$ and $\nu$ on $X$.
\end{rem}

We conclude the paper with the following

\begin{exm}{ubd}
Taking into account \PRO{C*} and item (a) of \COR{bd}, it is natural to ask whether the assumption
in this item that $\kkK$ is bounded is essential. Below we answer this issue in the affirmative.\par
Let $\Omega = \RRR$, $E = \KKK$ and let $\kkK$ consist of all functions $u \in C_0(\RRR,\KKK)$ for
which
\begin{equation*}
\sum_{n=1}^{\infty} \frac{u(n)}{2^n} = 0.
\end{equation*}
Observe that $\kkK$ is a closed proper linear subspace of $C_0(\RRR,\KKK)$ (as the kernel
of a continuous linear functional). However, invoking Tietze's extension theorem, it is an easy
exercise to show that $\kkK\bigr|_D = C(D,\KKK)$ for any compact set $D \subset \RRR$.
\end{exm}

\end{document}